\documentclass{amsart}
\usepackage{sansmath}
\usepackage{amssymb,amsthm,amsmath}
\usepackage[numbers,sort&compress]{natbib}
\usepackage{color}
\usepackage{graphicx}
\usepackage{stmaryrd}
\usepackage{tikz}
\usepackage{caption}
\usepackage[latin,english]{babel}
\usepackage{braket}
\usepackage[toc,page]{appendix}

\title[ ]{Reducibility  of relativistic Schr\"odinger equation with unbounded perturbations}

\thanks{$^{\dag}$ :Corresponding author.}
\author{Yingte Sun}
\address[Y. Sun]{School of Mathematical Sciences,
Yangzhou University,
yangzhou,
China} \email{sunyt15@fudan.edu.cn}
\author{Jing Li$^{\dag}$}
\address[J. Li]{School of Mathematics and Statistics,
Shandong University,
Weihai,
China} \email{xlijing@sdu.edu.cn}

\newcommand{\norm}[1]{\left\| #1 \right\|}
\newcommand{\abs}[1]{\left\lvert #1 \right\rvert}

\newcommand{\R}{\mathbb{R}}
\newcommand{\Z}{\mathbb{Z}}

\newcommand{\T}{\mathbb{T}}

\theoremstyle{plain}
\newtheorem{thm}{Theorem}[section]
 \newtheorem{cor}[thm]{Corollary}
 \newtheorem{lem}[thm]{Lemma}
 \newtheorem{prop}[thm]{Proposition}
 \theoremstyle{definition}
 \newtheorem{defn}[thm]{Definition}
 \theoremstyle{remark}
 \newtheorem{rem}[thm]{Remark}
 
 \numberwithin{equation}{section}

\begin{document}


\begin{abstract}
In this paper, we prove a reducibility result for a relativistic Schr\"odinger equation  on  torus with time quasi-periodic unbounded perturbations of order $1/2$. As far as we known, this is the first reducibility result for the relativistic Schr\"odinger equation.

\end{abstract}

\maketitle
\tableofcontents
\section{Introduction }

In this paper, we study  the reducibility of a relativistic Schr\"odinger equation with unbounded quasi-periodic perturbations on the torus $\mathbb{T}$,
\begin{equation}\label{eq1}
	\mathrm{i}\partial_{t}u=(-\partial_{xx}+\mathfrak{m}^2)^{\frac{1}{2}} u+\varepsilon \mathcal{W}(\omega t)u  \ , \quad x \in \T=\mathbb{R}/2\pi  \mathbb{Z},   \ \ t\in \R.
\end{equation}
The operator $(-\partial_{xx}+\mathfrak{m}^2)^{\frac{1}{2}}$, defined via its symbol $(\xi^2+\mathfrak{m}^2)^{\frac{1}{2}}$ under Fourier transform, is the kinetic energy operator of a relativistic particle of mass $\mathfrak{m}$, $0\leq \mathfrak{m} \leq \frac{1}{4}$. For more information about the  operator, we refer readers to \cite{Car1990}. The perturbation $\mathcal{W}(\omega t)$ is a pseudo-differential operator of order $\frac{1}{2}$, and quasi-periodic in time with frequencies $\omega\in \Omega=[1,2]^d$.
The goal of this paper is to find a bounded and time quasi-periodic transformation on $\mathrm{H}^r$ such that the original equation \eqref{eq1} can be transformeed into a block diagonal and time independent one.

In the context of linear PDEs, the reducibility theory concern the infinite dimensional systems  which are a diagonal operator under small perturbations of the form,
\begin{equation}\label{1.01}
\mathrm{i}\omega \cdot \partial_{\theta}+D+\varepsilon \mathcal{W}(\omega t), \quad \omega \in \mathbb{R}^d,
\end{equation}
where $D$ is a diagonal operator, $\varepsilon$ is small and $\omega$ is in some Cantor sets. In the past two decades, the reducibility problems of such systems have attracted lots of attentions and  can be divided into  two cases.  One is the diagonal operator with bounded perturbations, see \cite{L.M2019, B12, Ku93,Liang19, Wan08, W16}. The other is with unbounded perturbations, which is the focus of the present paper.

It is known that the reducibility procedure becomes more complicated in the case of unbounded perturbations. The first unbounded KAM reducibility result was obtained by Bambusi-Graffi \cite{Bam01}. Using Kuksin's lemma \cite{Ku97}, the authors were able to deal with the system \eqref{1.01}, where the  unperturbed part $D$ has order $n>1$ and the perturbation $\mathcal{W}(\omega t)$ is of order $ \delta < n-1$. The critical case $\delta =n-1$ was resolved by  Liu-Yuan \cite{Liu09}, which  greatly expanded the applications of classical unbounded KAM theorem. After that, the classical unbounded KAM theorem seems to have reached its limit.  The new breakthrough was obtained in \cite{Baldi2}. The authors  dealt with the system \eqref{1.01}, where the  unperturbed part $D$ has order $n=3$ and the perturbation $\mathcal{W}(\omega t)$ is also order $\delta=3$ . The  new strategy is to transform the original problem into the following new one
\begin{equation}
\mathrm{i} \omega \cdot \partial_{\theta}+D^++\varepsilon \mathcal{W}^+(\omega t), \quad \omega \in \mathbb{R}^d,
\end{equation}
using a series of bounded transformations before taking KAM iteration, where the new perturbation $\mathcal{W}^+$ is of low order. It is worth noting that the transformation  methods are completely different from different types of equation and perturbations.

In the past few years, people developed some sophisticated transformation techniques for different equations and unbounded perturbations. Bambusi et.al \cite{Bam18,Bam171,Bam19,Bam018} used the symbolic calculus of pseudo-differential operator to deal with harmonic oscillators under different unbounded perturbations. Feola-Greb\'ert \cite{F.G2019,F.G2020} studied the linear Schr\"odinger equation on Zoll manifolds  with different unbounded potentials. Using new coordinate transformation method, Bambusi-Langella-Montalto \cite{R.F2018}, Feola-Giuliani-Montalto-Procesi \cite{Bam019} gave a reducibility result for the linear transport equation under unbounded perturbations. Montolto \cite{Mo2019}, Sun et.al \cite{S19} studied the linear wave equation with some special unbounded perturbations. For more applications of these techniques in nonlinear PDEs, we refer readers to \cite{bal19,Baldi2,berti2,F.G20192,F.G20201,Mon2018}.

We also known the dimension of the space domains and the eigenvalues of the unperturbed part are also closely related to the reducibility process. They could induce new problems in estimating the number of non-resonance conditions. Interestingly, the strategy of reducing the order of perturbations is also a powerful tool to deal with this problem.  The idea is that the smoothing character of the perturbation can be used to recover a smoothness loss due to the small denominators. We refer readers to \cite{bal19,Bam019, F.G2020, Y2006,Mo2019}.

The main proof of this paper can be divided into two steps. At first, we take advantage of  the abstract pseudo-differential  operator technique in \cite{Bam17} to transform the original problem \eqref{eq1} into  a new problem \eqref{R1} such that the new perturbation  is sufficiently smooth. Then we apply the KAM technique  to obtain a reducibility result for  the equation \eqref{R1}. Comparing with the previous unbounded reducibility results, there are two troubles of the   eigenvalues
of the unperturbed part, i.e., the linear growth and the multiplicity of the eigenvalues, which
have seldom been dealt with before. The main difficulty is that
the eigenvalues become more sensitive to the unbounded perturbation when
they grow slowly at the linear rate. Furthermore, after the transformation in the first step,
the eigenvalues of the new unperturbed part contributes more resonances when
appearing multiple eigenvalues in the original unperturbed part.
For that reason, we take some reasonable restriction on the original perturbation (see Theorem \ref{THM1}), which can be
discarded if we take the length of the  torus as the extra parameters (see Theorem \ref{THM2}).
The idea is similar to that in \cite{R.F2018,Bam019} for the transport equation by regarding the constant vector field as the new parameters.
The main novelty of this paper is to further reveal the
relationship between the unperturbed part and the unbounded perturbation.
The results in this paper might be optimal for relativistic
Schrodinger equation on the torus if no  more assumptions of the perturbation is made.

\begin{rem}From the mathematical point of view, the reducibility result of equation \eqref{eq1} implies that the Sobolev norms of solutions stay bounded for all  time. In the context of non-small perturbations (without the small parameter $\varepsilon$), the dynamic behavior of  the solution of equation \eqref{eq1} is very rich. In \cite{Bam17, Mo20191}, the authors showed that if  $\omega$ satisfies some non-resonance conditions, then only a weak upper bound can be obtained, i.e., $\forall \epsilon \geq 0$, there exists a constant $C_{\epsilon}$ such that
\begin{equation}
\|u(t,x)\|_{_{\mathrm{H}^{}} } \leq C_{\epsilon}t^{\epsilon} \|u(0,x)\|_{_{\mathrm{H}^{r}} }.
\end{equation}
Furthermore, if $\omega$ is resonant, Maspero \cite{Mas18} constructed some perturbations which provoke polynomial growth of Sobolev norms. The conclusion in this paper is  supplement to the previous results. It further shows that stability of Sobolev norms is a non-resonant phenomenon.
\end{rem}

\begin{rem}In this paper, we use the abstract pdo (pseudo-differential operator) technique in \cite{Bam17} to regularize the perturbation, instead of the quantization technique in \cite{bal19}. The main advantage is that we can deal with much more general unbounded perturbations and even the high dimensional manifolds. Without taking much change, we can also deal with the following two models.

\textbf{1}: Relativistic Schr\"odinger equation on $\mathbb{S}^2$,
 \begin{equation}\label{1.4}
\mathrm{i}\partial_{t}u=\sqrt{-\Delta_{g}+\mathfrak{m}^2}u+\varepsilon [W(\omega t,x)(-\mathrm{i}\partial_{\phi})^{\frac{1}{2}}+V(\omega t ,x)]u, \quad u=u(t,x), \quad x\in \mathbb{S}^2.
\end{equation}
Here $\mathrm{i}\partial_{\phi}=\mathrm{i}(x_1\partial_{x_2}-x_2\partial_{x_1})$ is the $x_3$ component of the orbital angular momentum (and the generator of rotations about the $x_3$ axis).  Regarding more information about the perturbation, we refer readers to \cite{F.G2019}.

\textbf{2}: Relativistic Schr\"odinger equation on Zoll manifold of dimension $n \in \mathbb{N}$.
\begin{equation}\label{1.4}
\mathrm{i}\partial_{t}u=\sqrt{-\Delta_{g}+\mathfrak{m}^2}u+\varepsilon \mathcal{W}(\omega t)u, \quad u=u(t,x), \quad x\in \mathrm{M}^n.
\end{equation}
 Here $-\Delta_{g}$ is the positive Laplace-Beltrami operator on $\mathrm{M}^n$ and the linear operator  $\mathcal{W}(\omega t)$ is a time quasi-periodic pseudo-differential operator of order $0$ with frequency $\omega \in [1,2]^d$.

\end{rem}

The paper is organized as follows:
In section 2, we introduce some important notions  and definitions to precisely state our main results. In section 3, we introduce some infinite dimension matrix norm, such that the KAM process in section 5 can be well understood. In section 4, we introduce the abstract pseudo-differential operator (pdo) technique used in \cite{Bam019, Bam17}, such that  the original unbounded perturbation  can be reduced to a  smoothing operator. In section 5, we give a KAM reducibility result.
 In the Appendix A, we emphasize  the difference between relativistic Schr\"odinger equations on $\mathbb{T}$ and them on $\mathbb{T}_{\beta}$. In the Appendix B, we give some  important technical lemmas used in this paper.

Notations: In the present paper, we denote the notation $A\lesssim B$ as $A \leq CB$, where $C$ is a constant number depending on the fixed number $d,\mathfrak{m},s$.
\section{Main results}\label{sec:fun}
In order to state  the main results of the paper precisely, we introduce some  important notations and definitions in this section.
\subsection{Function space and pseudo-differential operators}\

Given any function $u \in L^2(\T)$, it can be expressed as
\begin{equation}
u(x)=\sum_{j\in \Z}\hat{u}(j)e^{\mathrm{i}j \cdot x }, \quad  \hat{u}(j)=\frac{1}{2\pi} \int_{\T} u(x)e^{-\mathrm{i} j \cdot x} dx.
\end{equation}

The Sobolev space on $\T$ is defined by
\begin{equation}
	\mathrm{H}^r(\T):=\Set{ u \in L^2(\T): |  \norm{u}_{\mathrm{H}^r(\T)}^2:= \sum_{j\in\Z}\braket{j}^{2r} \hat{u}(j)^2 <\infty  } ,
\end{equation}
where $\braket{j}=\max\{1,|j|\}$.

For a function $a:\T\times \Z \to \R$, define the difference operator $\Delta a(x,j) := a(x,j+1)-a(x,j)$ and let $\Delta^\beta=\Delta\circ...\circ\Delta$ be the composition  $\beta$ times of $\Delta$.
Then, we have the following definitions:
\begin{defn}(\cite{L.M2019}, Definition 2.1)\label{defn:symbols_semi}
	Let $m \in \mathbb{R}$, we say that a function $a:  \T \times \Z \to \R$ is a symbol of class $S^m$ if for any $j \in \Z$ the map  $ x\mapsto a( x, j)$ is smooth and for any
	$ \alpha, \beta \in \mathbb{N}$, there exists $C_{\alpha, \beta}>0$ such that
	$$
	\abs{\partial_x^\alpha \Delta^\beta a( x, j)} \leq C_{\alpha, \beta} \, \langle j \rangle^{m - \beta}  \ , \quad \forall x \in \T   \ .
	$$
\end{defn}
\begin{defn}(\cite{L.M2019}, Definition 2.2)\label{defn:symbols20}
	Given a symbol $a \in S^m$, we say that  $Op(a) \in OPS^m$  is the  associated pseudo-differential operator of $a$ if for any $u \in L^2(\T)$
\begin{equation}
Op(a)[u](x)=\sum_{j \in \Z} a(x, j) \hat{u}(j) e^{\mathrm{i} j \cdot x}.
\end{equation}
\end{defn}
We  endow the operator $Op(a)\in OPS^m$ a family of seminorms
$$\chi^m_\rho(Op(a)):=\sum_{\alpha+\beta \leq \rho} \sup_{x \in \mathbb{T},j \in \mathbb{Z}} \langle j\rangle^{-m+\beta}|\partial^{\alpha}_{x}\Delta^\beta a(x,j)|, \ \rho\in \mathbb{N}_0.$$

\begin{defn}\label{defn:symbols3}
Consider the  pseudo-differential operator $A(\theta)$ depending the angle variable $\theta \in \mathbb{T}^d$ in a smooth way. Then the operator $A(\theta)$ can be expressed as
\begin{equation}
A(\theta)=\sum_{\ell \in \mathbb{Z}^d}\hat{A}(\ell) e^{\mathrm{i}\ell \cdot \theta},
\end{equation}
where $\hat{A}(\ell) \in OPS^m$. We denote $A(\theta)$ by $C^{\infty}(\mathbb{T}^d,OPS^m)$. If the operator $A(\theta)$ is also Lipschitz-way depending on the parameter $\omega \in \Omega \subseteq \mathbb{R}^d$, we denote the set of all these $A(\theta,\omega)$  by $\mathcal{L}ip(\Omega,C^{\infty}(\mathbb{T}^d,OPS^m))$.
\end{defn}
\begin{rem}The symbol of the pseudo-differential operator $A(\theta)$ can be expressed as
\begin{equation}
a(\theta,x,j)=a(x,j)(\ell)e^{i\ell\cdot \theta},
\end{equation}
where $a(x,j)(\ell)$ is the symbol of the pseudo-differential operator $\hat{A}(\ell)$.
\end{rem}
\begin{defn}
Let $s >\frac{d}{2}$, the operator $A(\theta)\in C^{\infty}(\mathbb{T}^d,OPS^m)$ can be endowed a family of seminorms:
\begin{equation}
\chi^m_{\rho,s}(A(\theta)):=\Big( \sum_{\ell \in \mathbb{Z}^d}\langle \ell \rangle^{2s}(\chi^m_{\rho}(\hat{A}(\ell)))^2\Big)^{\frac{1}{2}}, \ \rho\in \mathbb{N}_0.
\end{equation}
Moreover, we can endow the operator $A(\theta,\omega) \in \mathcal{L}ip(\Omega,C^{\infty}(\mathbb{T}^d,OPS^m))$ a family of Lipschitz seminorms:
\begin{align}
\chi^{m,\mathcal{L}ip,\Omega}_{\rho,s}(A(\theta,\omega)):&=\chi^{m,\sup,\Omega}_{\rho,s}(A(\theta,\omega))+\chi^{m,lip,\Omega}_{\rho,s}(A(\theta,\omega))\\
&=\sup_{\omega \in \Omega}\chi^{m}_{\rho,s}(A(\theta,\omega))+\sup_{\omega_1,\omega_2\in \Omega}\frac{\chi^{m}_{\rho,s}( A(\omega_1)-A(\omega_2))}{|\omega_1-\omega_2|}.
\end{align}
\end{defn}

\subsection{Main results}\

The perturbation $\mathcal{W}(\omega t)$ is a quasi-periodic driving pseudo-differential operator, which satisfies the following two conditions:

\textbf{(C1)}:  $\mathcal{W}(\omega t)$ is an Hermitian operator, and belongs to $C^{\infty}(\mathbb{T}^d,OPS^{\frac{1}{2}})$.

\textbf{(C2)}: Set the symbol of  pseudo-differential operator $\mathcal{W}(\omega t)$ as $w(\theta,x,j)$, one has
$$ \int_{\mathbb{T}^d}\int_{\mathbb{T}}w(\theta,x,j)dxd\theta=a\langle j\rangle^{\frac{1}{2}}+b(j), \quad j \in \mathbb{Z},$$
where $a$ is independent of $j$ and $b$ is dependent on $j$.  Moreover, there exists an absolute constant $C$ such that
$$b(j) \leq C, \quad  \forall j \in \mathbb{Z}.$$

\begin{thm}\label{THM1} Consider the equation \eqref{eq1} and assume conditions \textbf{(C1)} and \textbf{(C2)}. For any $r\geq 0$, there exists $\varepsilon^* >0$, such that for any $0<\varepsilon< \varepsilon^*$, there
exists a closed asymptotically full Lebesgue  set $\Omega_{\varepsilon} \subseteq \Omega:=[1,2]^{d}$.  For each $ \omega \in \Omega_{\varepsilon}$, there exist a family of linear and invertible bounded operator $\mathcal{U}(\theta,\omega) \in \mathcal{L}(\mathrm{H}^r)$ conjugate the equation \eqref{eq1} to
\begin{equation}
\mathrm{i}\partial_{t}u=\mathbf{H}^{\infty}u, \quad \mathbf{H}^{\infty}=\mathrm{diag}\Big\{\Lambda^{\infty}_j(\omega) \big| \ j\in \mathbb{N}\Big\}.
\end{equation}
Here $\Lambda^{\infty}_j,j\geq 1$ is a $2 \times 2$   Hermitian matrix, and $\Lambda^{\infty}_0$ is a real number close to $\mathfrak{m}$.
\end{thm}

As a consequence, we can get a Sobolev norms control of the flow generated by the equation \eqref{eq1}.

\begin{cor}\label{cor1}For any $r \geq 0$ and $\omega \in \Omega_{\varepsilon} $, the solution $u(t,x)$ of equation \eqref{eq1} with initial condition $u(0,x) \in \mathrm{H}^{r}$ satisfies
\begin{equation}
c_s\|u(0,x)\|_{\mathrm{H}^{r}} \leq \|u(t,x)\|_{\mathrm{H}^{r}} \leq C_{s} \|u(0,x)\|_{\mathrm{H}^{r}}.
\end{equation}
\end{cor}

 We emphasized that the condition \textbf{(C2)} is indispensable. Inspired by \cite{s19}, the author obtained a family of analytical solutions of elliptic equation  by taking the length of space torus as frequency parameters.  We can also introduce the length of space torus as frequency parameters to discard the condition \textbf{(C2)}.
Hence, we  consider the  following equation:
\begin{equation}\label{eq2}
	\mathrm{i}\partial_{t}u=(-\partial_{xx}+\mathfrak{m}^2)^{\frac{1}{2}} u+\varepsilon \mathcal{W}(\omega t)u  \ , \quad x \in \T_{\beta}=\mathbb{R}/2\pi \beta \mathbb{Z},   \ \ t\in \R \ ,
\end{equation}
where $\mathcal{W}(\omega t)$ is a pseudo-differential operator of order $\frac{1}{2}$, and quasi-periodic in time with frequencies $\omega \in [1,2]^d$. The space domain changes with the parameter $\beta \in [\frac{1}{2},1]$.

Then, we can prove the following reducibility result.
\begin{thm}\label{THM2} Let $\mathcal{W}(\omega t) $ be an Hermitian operator and belongs to $ C^{\infty}(\mathbb{T}^d_{\beta},OPS^{\frac{1}{2}})$.  For any $r\geq 0$, there exist $\varepsilon^* >0$, such that for any $0<\varepsilon< \varepsilon^*$, there
exists a closed asymptotically full Lebesgue  set $\tilde{\Omega}_{\varepsilon} \subseteq \tilde{\Omega}:=[1,2]^{d+1}$.  For each $ \tilde{\omega }:=(\omega,\frac{1}{\beta})\in \tilde{\Omega}_{\varepsilon}$, there exist a family of time quasi-periodic and invertible bounded operator $\mathcal{U}(\theta,\tilde{\omega}) \in \mathcal{L}(\mathrm{H}^r)$ conjugate the equation \eqref{eq2} to
\begin{equation}
\mathrm{i}\partial_{t}u=\mathbf{H}^{\infty}u, \quad \mathbf{H}^{\infty}=\mathrm{diag}\Big\{ \Lambda^{\infty}_j(\tilde{\omega})\big| \ j\in \mathbb{ N}\Big\}.
\end{equation}
 Here $\Lambda^{\infty}_j,j\geq 1$ is a $2 \times 2$   Hermitian matrix, and $\Lambda^{\infty}_0$ is a real number close to $\mathfrak{m}$.
\end{thm}

\begin{rem}The proof of Theorem \ref{THM2} is essentially the same to that of Theorem \ref{THM1}. The main differences are explained in detail in Appendix A.
\end{rem}
\section{Matrix representation of linear operator }\label{sec:block}\

Let $\mathrm{H}^\infty:=\cap_{r\in\R}\mathrm{H}^r$ and $\mathrm{H}^{-\infty}:=\cup_{r\in\R}\mathrm{H}^r$.
For any linear operator $A\colon \mathrm{H}^\infty \to \mathrm{H}^{-\infty}$ , we  take its matrix representation of block coefficients $(A_{[m]}^{[n]})_{m,n \in \mathbb{N}} $  as
\begin{equation}\label{3.11}
A_{[m]}^{[n]}=\left(
\begin{array}{cc}  
A_m^n& A_m^{-n} \\
A_{-m}^n & A_{-m}^{-n}
\end{array}
\right )
\end{equation}
on the basis $(\hat{e}_j:=e^{\mathrm{i}jx})_{j \in \mathbb{Z}}$,  defined for $m,n\in \mathbb{Z}$. Here, $A_m^n$ is defined by
$$
 A_m^n \equiv  \langle A  \hat{e}_m , \hat{e}_n \rangle_{\mathcal{H}^0}  \ .
$$
The matrix $ A_{[m]}^{[n]}$ can be seen as a liner operator in $\mathcal{L}(E_m,E_n)$ for any $m.n \in \mathbb{N}$, where $E_m$ is defined as
\begin{equation}
E_m:=\mathrm{span}\{e^{\mathrm{i}mx},e^{-\mathrm{i}mx}\}.
\end{equation}

In this paper we also consider the $\theta$-depending linear operator
$$\mathbb{T}^d\ni \theta\mapsto A:=A(\theta)=\sum_{\ell \in \mathbb{Z}^d}\hat{A}(\ell)e^{\mathrm{i}\ell \cdot  \theta},$$
where $\hat{A}(\ell) \in \mathcal{L}(\mathrm{H}^\infty, \mathrm{H}^{-\infty})$ . Then $A(\theta)$ can be regarded as an operator acting on function $u(\theta,x)$ of space-time as
$$(Au)(\theta,x)=(A(\theta)u(\theta,\cdot))(x).$$

Having the infinite dimensional matrix $A$ and $A(\theta)$, we can define the following $s$-decay norms.
\begin{defn}\label{Def1}(s-decay norm)

\textbf{I:}
The $s$-decay norms of infinite dimensional matrix $ A$  is  defined by
\begin{equation}
\|A\|_{s,s}=\Big(\sum_{h \in \mathbb{N}}\langle h\rangle^{2s} \sup_{|i-j|=h} \|A_{[j]}^{[i]}\|^2\Big)
^{\frac{1}{2}},
\end{equation}
where  $\|A_{[j]}^{[i]}\|$ is the $\mathcal{L}^2$ operator norm of $\mathcal{L}(E_j,E_i)$.\

\textbf{II:}
Considering a $\theta$ -depending  infinite dimensional matrix $ A(\theta)$, we define its norms as
\begin{equation}
\| A(\theta)\|^{s}_{s,s}=\Big(\sum_{\ell \in \mathbb{Z}^d,h \in \mathbb{N}}\langle \ell,h \rangle^{2s} \sup_{|i-j|=h} \|\hat{A}_{[j]}^{[i]}(\ell)\|^2\Big)^{\frac{1}{2}},
\end{equation}
where $\langle \ell,h \rangle=\max\{|\ell|,|h|,1\}$. We denote by $\mathcal{M}^{s}$ the space of matrices  with finite $s$-decay norm.\

\textbf{III:} If the linear operator $A(\theta)$ is a family Lipschitz map  from $\R^{d}\supseteq\Omega\ni \omega $ to $\mathcal{M}^{s}$, we define the Lipschitz $s$-decay norm as
\begin{align}\label{6.1}
\| A(\theta)\|^{s,\mathcal{L}ip,\Omega}_{s,s}=\sup_{\omega \in \Omega}\| A(\omega)\|^{s}_{s,s}+\sup_{\omega_1,\omega_2\in \Omega}\frac{\| A(\omega_1)-A(\omega_2)\|^{s}_{s,s}}{|\omega_1-\omega_2|}.
\end{align}
We denote by $\mathcal{M}^{s,\mathcal{L}ip,\Omega}$ the family Lipschitz map  from $\R^{d}\supseteq\Omega\ni \omega $ to $\mathcal{M}^{s}$ with finite Lipschitz $s$-decay norm. For notionally convenience, drop the range of $\omega$, $\mathcal{M}^{s,\mathcal{L}ip,\Omega}$ denoted as $\mathcal{M}^{s,\mathcal{L}ip}$.\
\end{defn}

\begin{rem}\label{3.20} In the present paper, we claim that the $\theta$-depending linear operator $A(\theta)$ is an Hermitian operator, if and only if
\begin{align}
A=A^*&\Leftrightarrow \hat{A}(\ell)^*=\hat{A}(-\ell), \ \forall \ell \in \mathbb{Z}^d \Leftrightarrow (\hat{A}^{[m]}_{[n]}(-\ell))^*=\hat{A}^{[n]}_{[m]}(\ell), \  \forall \ell \in \mathbb{Z}^d,m,n \in \mathbb{N}.
\end{align}

\end{rem}
It is crucial to investigate the tame  or algebra property  of $s$-decay norm. Thus, we need the following Lemmas.
\begin{lem}(\cite{berti}, Lemma 2.6, Lemma 2.7 and Lemma 2.8)\label{2.1}\

\textbf{A:} If $\mathfrak{s} \geq \mathfrak{s}_0>\frac{1}{2}$,
there is a constant $C(\mathfrak{s} )$  such that
\begin{equation}
\|Au\|_{\mathrm{H}^\mathfrak{s}} \leq C(\mathfrak{s} ) (\|A\|_{\mathfrak{s},\mathfrak{s}}\|u\|_{\mathrm{H}^{\mathfrak{s}_0}}+ \|A\|_{\mathfrak{s}_0,\mathfrak{s}_0}\|u\|_{\mathrm{H}^\mathfrak{s}}).
\end{equation}

For any $s\geq \mathrm{s}_0>\frac{d+1}{2}$, the following results hold:

\textbf{B:} there is a constant $C(s)$ such that
\begin{equation}\label{3.3}
\| AB(\theta)\|^{s}_{s,s} \leq C(s)( \| A\|^{\mathrm{s}_0}_{\mathrm{s}_0,\mathrm{s}_0}\| B\|^{s}_{s,s}+\| A\|^{s}_{s,s}\| B\|^{\mathrm{s}_0}_{\mathrm{s}_0,\mathrm{s}_0}).
\end{equation}

\textbf{C:} given an infinite dimension matrix $A(\theta)$ , for any $N \in \mathbb{N}$, we define the cutoff matrix $\Pi_NA$ as
$$
(\Pi_NA)^{[i]}_{[j]}(\ell)=\begin{cases}
\hat{A}^{[i]}_{[j]}(\ell), \quad if \ |i-j|< N \ and \  |\ell| < N,\\
0, \quad \quad \ \  \ \ otherwise.\\
\end{cases}$$
Denote $\Pi^\perp_{N}A$ as $A-\Pi_{N}A$, we have
\begin{equation}\label{3.4}
\|\Pi^\perp_{N}A\|^{s}_{s,s} \leq CN^{-\beta}\|A\|^{s+\beta}_{s+\beta,s+\beta},
\end{equation}
\begin{equation}\label{3.5}
\|\Pi_{N}A\|^{s}_{s,s}, \ \|\Pi^\perp_{N}A\|^{s}_{s,s} \leq \|A\|^{s}_{s,s}.
\end{equation}
The bounds of \eqref{3.3},\eqref{3.4},\eqref{3.5} are valid by replacing $\|\cdot\|^{s}_{s,s}$ by $\|\cdot\|^{s,\mathcal{L}ip}_{s,s}$.
\end{lem}

\begin{lem}\label{2.101}(\cite{Baldi2}, Lemma 2.4) Let $\mathrm{s}_0>\frac{d+1}{2}$, one has
\begin{equation}
\|A(\theta)\|_{s,s} \leq C(s)\|A(\theta)\|^{s+\mathrm{s}_0}_{s+\mathrm{s}_0,s+\mathrm{s}_0}.
\end{equation}
Here $\|A(\theta)\|_{s,s}=\Big(\sum_{h \in \mathbb{N}}\langle h\rangle^{2s} \sup \limits_{|i-j|=h \atop \theta \in \mathbb{T}^d } \|A_{[j]}^{[i]}(\theta)\|^2\Big)^{\frac{1}{2}}$.
\end{lem}
\begin{rem}\label{2.102}From Lemma \ref{2.1}, we see that a  linear operator $\mathrm{A}\colon \mathrm{H}^\infty \to \mathrm{H}^{-\infty}$ with finite $\mathfrak{s}$-decay norms ($\mathfrak{s}>\frac{1}{2}$) is a bounded operator from $\mathrm{H}^{\mathfrak{s}}$ to $\mathrm{H}^{\mathfrak{s}}$. Actually,  the linear operator $\mathrm{A}$ can be extended to  a  bounded operator from $\mathrm{H}^{r}$ to $\mathrm{H}^{r}$ with $r \in [0,\mathfrak{s}]$. From tame estimate in Lemma 6.1 \cite{Baldi2}, one can get  quantitative bounds $\|A\|_{\mathcal{L}(\mathrm{H}^r)}\leq C_{r,\mathfrak{s}}\|A\|_{\mathfrak{s},\mathfrak{s}}$.
\end{rem}
In the KAM procedure of section 4, the smoothing operator plays an important role. Hence, we introduce the following norms.
\begin{defn}\label{3.40}
Considering a  time quasi-periodic linear operator $A(\theta)$, we introduce a new $s$-decay norm as
\begin{equation}
\| A(\theta)\|^{s}_{s+m,s+n}=\Big(\sum_{\ell \in \mathbb{Z}^d,h \in \mathbb{N}}\langle \ell,h\rangle^{2s}\sup_{|i-j|=h} \langle i \rangle^{2n}\|\hat{A}_{[j]}^{[i]}(\ell)\|^2\langle j \rangle^{-2m}\Big)^{\frac{1}{2}}.
\end{equation}
We denote $\mathcal{M}^{s}_{s+m,s+n}$ as the  space of  matrices with  finite $s$-decay norm.
Moreover, if the linear operator $A(\theta)$ is a family of  Lipschitz map  from $\R^{d}\supseteq\Omega\ni \omega $ to $\mathcal{M}^{s}_{s+m,s+n}$, we  can define the Lipschitz $s$-decay norm in the same way as Definition \ref{Def1}, III.
\end{defn}

\begin{rem}\label{2.3}
Define a $\theta$-independent diagonal operator $D$, acting on $u \in \mathrm{H}^{0}$ as
$$D u(x)=\sum_{k\in\Z}\langle k \rangle\hat{u}_ke^{\mathrm{i}kx}.$$
For any $m,n \in \mathbb{R} $, $A(\theta) \in \mathcal{M}^{s}_{s+m,s+n}$, there exists a linear operator $Q(\theta) \in \mathcal{M}^{s}_{s,s}$ such that $\hat{Q}^{[i]}_{[j]}(\ell)=\frac{\hat{A}^{[i]}_{[j]}(\ell)\langle i\rangle^n}{\langle j\rangle^m}$. Moreover, one can obtain
$$\|A(\theta)\|^{s}_{s+m,s+n}=\|\langle D \rangle ^{-n}Q(\theta) \langle D \rangle ^{m}\|^{s}_{s+m,s+n}=\|Q(\theta)\|^{s}_{s,s}.$$
\end{rem}
\begin{lem}\label{3.7}
Fix $s\geq \mathrm{s}_0> \frac{d+1}{2}$. For any linear operator $A \in \mathcal{M}^{s}_{s+m,s+l}$  and $B \in  \mathcal{M}^{s}_{s+l,s+n}$, there exists a constant $C:=C(s)$ such that
\begin{equation}
\|AB\|^{s}_{s+m,s+n} \leq C(s)\Big(\|A\|^{s}_{s+m,s+l} \|B\|^{\mathrm{s}_0}_{\mathrm{s}_0+l,\mathrm{s}_0+n}+\|A\|^{\mathrm{s}_0}_{\mathrm{s}_0+m,\mathrm{s}_0+l} \|B\|^{s}_{s+l,s+n}\Big).
\end{equation}
The assertion holds true by replacing $\|\cdot\|^{s}_{s+m,s+n}$ by $\|\cdot\|^{s,\mathcal{L}ip}_{s+m,s+n}$.
\end{lem}
\begin{proof} These bounds can be obtained  from  Lemma \ref{2.1} and Remark \ref{2.3}.
\end{proof}

\begin{lem} \label{2.4}Assume that $\mathrm{s}_0 >\frac{d+1}{2}$ and  $C(s)\|A\|^{\mathrm{s}_0,\mathcal{L}ip}_{\mathrm{s}_0+m,\mathrm{s}_0+m} \leq \frac{1}{2}$ for some $m \in \mathbb{R}$ and large $C(s)>0$ depending on $s \geq \mathrm{s}_0$, then the map $\Phi:=\mathrm{Id}+\Psi$ defined as  $\Phi=e^{\mathrm{i}A}=\sum_{p \geq 0}\frac{1}{p!}(\mathrm{i}A)^{p}$ satisfies
\begin{equation}\label{3.8}
\|\Psi\|^{s,\mathcal{L}ip}_{s+m,s+m} \leq C \|A\|^{s,\mathcal{L}ip}_{s+m,s+m},
\end{equation}
where $C$ is a constant depending on $s,d,m$.
\end{lem}
\begin{proof} From Lemma \ref{3.7}, for some $C(s)\geq 0$,
\begin{equation}
\|A^n\|^{s,\mathcal{L}ip}_{s+m,s+m} \leq n[C(s)\|A\|^{\mathrm{s}_0,\mathcal{L}ip}_{\mathrm{s}_0+m,\mathrm{s}_0+m}]^{n-1}C(s)\|A\|^{s,\mathcal{L}ip}_{s+m,s+m}.
\end{equation}
Hence,
\begin{equation}
\|\Psi\|^{s,\mathcal{L}ip}_{s+m,s+m} \leq \|A\|^{s,\mathcal{L}ip}_{s+m,s+m} \sum_{p \geq 1}\frac{C(s)^p}{p!} (\|A\|^{\mathrm{s}_0,\mathcal{L}ip}_{\mathrm{s}_0+m,\mathrm{s}_0+m})^{p-1}
\end{equation}
for some large $C(s)>0$. The bounds  \eqref{3.8} can be obtained from the small condition of  $C(s)\|A\|^{\mathrm{s}_0,\mathcal{L}ip}_{\mathrm{s}_0+m,\mathrm{s}_0+m}$.
\end{proof}

\section{Reduction of the order of perturbations}
The main goal of this section is to conjugate the original problem \eqref{eq1} to a new one, which the new perturbation is a sufficiently smoothing  operator. By direct calculation,
the equation \eqref{eq1} can be  rewrited as
\begin{equation}
\mathrm{i}\partial_{t}u=  \mathcal{K} u+ \mathcal{Q}u+\varepsilon \mathcal{W}(\omega t)[u] ,
\end{equation}
where $\mathcal{K}=(-\partial_{xx})^{\frac{1}{2}}$, $\mathcal{K}e^{\mathrm{i}jx}=|j|e^{\mathrm{i}jx},\forall j \in \mathbb{Z}$. We remark that $\mathcal{Q}$ is a  pseudo-differential operator of order $-1$ and  give a simple proof in Lemma \ref{OPS}. Moreover, we know that
$$\mathcal{Q}e^{\mathrm{i}jx}=\frac{c(\mathfrak{m},|j|)}{\langle j \rangle }e^{\mathrm{i}j\cdot x},$$
 where $c(\mathfrak{m},|j|)$ depends on $\mathfrak{m},j$ and $c(\mathfrak{m},|j|) \leq \mathfrak{m}^2$.

\begin{lem}
Given a linear operator $\mathcal{Z}:\mathrm{H}^{\infty}\mapsto \mathrm{H}^{-\infty}$, if $[\mathcal{Z}, \mathcal{K}]=0$,  the block matrix representation of $\mathcal{Z}$ satisfies
\begin{equation*}\label{Z-1}
\mathcal{Z}^{[i]}_{[j]}=0,   \quad \forall i\neq j.
\end{equation*}
\end{lem}
\begin{proof}
From $[\mathcal{Z}, \mathcal{K}]=0$, for any $i,j \in \mathbb{N}$, one gets that
\begin{equation}\label{Z1}
\mathcal{Z}^{[i]}_{[j]}(i-j)=0.
\end{equation}
Hence, for any $i \neq j$, \eqref{Z1} implies that
\begin{equation*}
\mathcal{Z}^{[i]}_{[j]}=0.
\end{equation*}

\end{proof}

\begin{lem}
Given a pseudo-differential operator $\mathcal{B} \in OPS^{\eta}$, the corrseponding linear operator $e^{\mathrm{i}\kappa\cdot \mathcal{K}} \mathcal{B} e^{-\mathrm{i}\kappa\cdot \mathcal{K}}$ is $2\pi$ periodic to $\kappa$.
\end{lem}
\begin{proof}
The spectrum of $\mathcal{K}$ is integer,  thus $e^{\mathrm{i}\kappa\cdot \mathcal{K}}=e^{\mathrm{i}(\kappa+2\pi)\cdot \mathcal{K}}$.
\end{proof}

The following Lemma plays an important role in the regularization process.
\begin{lem}\label{4.3}
Take the Cantor set $\Omega_{0,\alpha} \subseteq \Omega$ as
\begin{equation}\label{4.30}
\Omega_{0,\alpha}:=\Big\{\omega \in \Omega: |\omega \cdot\ell+ \mathrm{m}| \geq \frac{\alpha}{1+|\ell|^{d+2}}, \quad \forall (\ell, \mathrm{m}) \in \mathbb{Z}^{d+1}\setminus \{0\} \Big\}.
\end{equation}
Let $\mathcal{W}$ be an Hermitian operator and belongs to $ \mathcal{L}ip (\Omega,C^{\infty}(\mathbb{T}^d, OPS^{\eta})), \eta \leq 1$.
Then, the homological equation
\begin{equation}\label{h3}
\omega \cdot \partial_{\theta}  \mathcal{B}+\mathrm{i}[\mathcal{K},\mathcal{B}]=\mathcal{W}-\langle \mathcal{W} \rangle
\end{equation}
with
\begin{equation}\label{4.07}
\langle \mathcal{W} \rangle:= \frac{1}{(2\pi)^{d+1}}\int_{\mathbb{T}^d}\int_{\mathbb{T}}e^{\mathrm{i}\kappa\cdot \mathcal{K}} \mathcal{W} e^{-\mathrm{i}\kappa\cdot \mathcal{K}}d\kappa d\theta
\end{equation}
has a solution $\mathcal{B} \in \mathcal{L}ip (\Omega_{0,\alpha},C^{\infty}(\mathbb{T}^d, OPS^{\eta}))$. Moreover, the operator $\mathcal{B}$ is an Hermitian operator too.
\end{lem}
\begin{proof}
For any $ \mathcal{W} (\theta) \in \mathcal{L}ip (\Omega,C^{\infty}(\mathbb{T}^d, OPS^{\eta}))$, we define $\mathcal{W} (\theta,\kappa)=e^{\mathrm{i}\kappa\cdot \mathcal{K}} \mathcal{W}(\theta) e^{-\mathrm{i}\kappa\cdot \mathcal{K}}$. From Remark \ref{8.4}, we know that
 $$\mathcal{W} (\theta,\kappa) \in \mathcal{L}ip (\Omega,C^{\infty}(\mathbb{T}^{d+1}, OPS^{\eta})).$$
 Since $\mathcal{W} (\theta,\kappa)$ is defined on $\T^{d+1}$, it can be expanded by its Fourier series as
$$\mathcal{W} (\theta,\kappa)=\sum_{(\ell,\mathrm{m}) \in \Z^{d+1}}\hat{\mathcal{W}}_{\ell,\mathrm{m}}e^{\mathrm{i}( \ell \cdot\theta+\mathrm{m} \cdot \kappa)},$$
where
$$\mathcal{W} (\theta)=\mathcal{W} (\theta,0)=\sum_{(\ell,0) \in \Z^{d+1}}\hat{\mathcal{W}}_{\ell,0}e^{\mathrm{i} \ell \cdot \theta}.$$
The homological equation \eqref{h3} can be extended as
\begin{equation}\label{h2}
\omega \cdot \partial_{\theta}  \mathcal{B}(\theta,\kappa)+\mathrm{i}[\mathcal{K}, \mathcal{B}(\theta,\kappa)]=\mathcal{W}(\theta,\kappa)-\langle \mathcal{W}(\theta,\kappa) \rangle.
\end{equation}
Obviously, if $\mathcal{B}(\theta,\kappa)$ is the solution of equation \eqref{h2}, then $\mathcal{B}(\theta,0)$ is the solution of equation \eqref{h3}. Notice that
\begin{align*}
\mathrm{i}[\mathcal{K}, \mathcal{B}(\theta,\kappa)]&=\frac{d}{ds}\Big|_{s=0} e^{\mathrm{i}s\cdot \mathcal{K}} \mathcal{B}(\theta,\kappa) e^{-\mathrm{i}s \cdot \mathcal{K}}\\
&=\frac{d}{ds}\Big|_{s=0}\mathcal{B}(\theta,\kappa+s)=\sum_{(\ell,\mathrm{m}) \in\Z^{d+1}}\hat{\mathcal{B}}_{\ell,\mathrm{m}} \frac{d}{ds}\Big|_{s=0} e^{\mathrm{i} \ell \cdot \theta+\mathrm{i}\mathrm{m}\cdot(\kappa+s)}\\
&=\sum_{(\ell,\mathrm{m}) \in\Z^{d+1}} \mathrm{i} \mathrm{m} \hat{\mathcal{B}}_{\ell,\mathrm{m}} e^{\mathrm{i} \ell \cdot \theta+\mathrm{i}\mathrm{m} \cdot \kappa}.
\end{align*}
The homological equation \eqref{h2} is equivalent to
\begin{equation}
\mathrm{i}(\omega\cdot \ell+\mathrm{m})\hat{\mathcal{B}}_{\ell,\mathrm{m}}=\hat{\mathcal{W}}_{\ell,\mathrm{m}}, \quad  (\ell,\mathrm{m}) \neq (0, 0) \quad and \quad \hat{\mathcal{B}}_{0,0}=0.
\end{equation}

Since the operator $\mathcal{W}(\theta,\kappa)$  belongs to $\mathcal{L}ip (\Omega,C^{\infty}(\mathbb{T}^{d+1}, OPS^{\eta}))$, the seminorms of $\hat{\mathcal{W}}_{\ell,\mathrm{m}}$  decay faster than any power of $|\ell|+|\mathrm{m}|$. From the definition of $\Omega_{0,\alpha}$, we see that $\hat{\mathcal{B}}_{\ell,\mathrm{m}}$ also decay faster than any power of $|\ell|+|\mathrm{m}|$. Observing that $\mathcal{B}(\theta)=\mathcal{B}(\theta,0)$, thus $\mathcal{B}(\theta) \in C^{\infty}(\mathbb{T}^d,OPS^{\eta})$.

Furthermore, for any $\omega_1,\omega_2 \in \Omega_{0,\alpha}$, one has
\begin{align}
\hat{\mathcal{B}}_{\ell,\mathrm{m}}(\omega_1)-\hat{\mathcal{B}}_{\ell,\mathrm{m}}(\omega_2)=\frac{\hat{\mathcal{W}}_{\ell,\mathrm{m}}(\omega_1)[(\omega_2-\omega_1)\ell]}{\mathrm{i}(\omega_1 \ell+\mathrm{m})(\omega_2 \ell+\mathrm{m})}+\frac{\hat{\mathcal{W}}_{\ell,\mathrm{m}}(\omega_1)-\hat{\mathcal{W}}_{\ell,\mathrm{m}}(\omega_2)}{\mathrm{i}(\omega_2 \ell+\mathrm{m})}.
\end{align}
Hence, from the non-resonance condition \eqref{4.30}, we can obtain the Lipschitz regular of $\mathcal{B}$ to the parameter $\omega$.

Moreover, from
$$\mathcal{W}- \mathcal{W}^*=e^{-\mathrm{i}\kappa\cdot \mathcal{K}} (\mathcal{W}(\theta,\kappa)- \mathcal{W}^*(\theta,\kappa))e^{\mathrm{i}\kappa\cdot \mathcal{K}}, \quad \mathcal{B}- \mathcal{B}^*=e^{-\mathrm{i}\kappa\cdot \mathcal{K}} (\mathcal{B}(\theta,\kappa)- \mathcal{B}^*(\theta,\kappa))e^{\mathrm{i}\kappa\cdot \mathcal{K}},$$
we know that $\mathcal{W}$(resp $\mathcal{B}$) is an Hermitian operator, if and only if  $\mathcal{W}(\theta,\kappa)$(resp $\mathcal{B}(\theta,\kappa)$) is an Hermitian operator.
From $\hat{\mathcal{B}}_{\ell,\mathrm{m}}=\frac{\hat{\mathcal{W}}_{\ell,\mathrm{m}}}{\mathrm{i}(\omega\cdot \ell+\mathrm{m})}$ and Remark \ref{3.20}, we obtain that $\mathcal{B}$ is an Hermitian operator.
\end{proof}

\begin{thm}\label{3.0}
For any $M> 0$, there exists a sequence of symmetric maps $\{\mathcal{B}_i(\theta,\omega)\}^M_{i=0}$ with $\mathcal{B}_i(\theta,\omega) \in \mathcal{L}ip (\Omega_{0,\alpha},C^{\infty}(\mathbb{T}^d, OPS^{\frac{1}{2}-\frac{1}{2}i}))$ such that the change of variables $$u=e^{-\varepsilon \mathrm{i} \mathcal{B}_0(\theta,\omega)} \cdots e^{-\varepsilon \mathrm{i} \mathcal{B}_M(\theta,\omega)}v$$ conjugates  the Hamiltonian $\mathcal{H}_0=\mathcal{K}+\mathcal{Q} + \varepsilon\mathcal{W }(\omega t)$ to
\begin{equation}
\mathcal{H}_{M+1}= \mathcal{K}+\mathcal{Q}+\varepsilon \mathcal{Z}^{M+1}+\varepsilon \mathcal{W}^{M+1},
\end{equation}
where $\mathcal{Z}^{M+1}$ is time-independent and fulfils
\begin{equation}\label{3.10}
[\mathcal{Z}^{M+1},\mathcal{K}]=0.
\end{equation}
Also,
\begin{align}\label{3.1}
\mathcal{Z}^{M+1}(\omega) &\in  \mathcal{L}ip (\Omega_{0,\alpha}, OPS^{\frac{1}{2}}),\\
\label{3.2}
\mathcal{W}^{M+1}(\theta, \omega) &\in \mathcal{L}ip (\Omega_{0,\alpha},C^{\infty}(\mathbb{T}^d, OPS^{-\frac{1}{2}M})).
\end{align}
Furthermore, $\mathcal{Z}^{M+1},\mathcal{W}^{M+1}$ are Hermitian operators.
\end{thm}
\begin{proof}
We prove this theorem by the induction method.

For $i=0$,  the hypotheses are verified for  $\mathcal{Z}^0=0$, $\mathcal{W}^{0}=\mathcal{W}$.

Moreover, suppose that $\mathcal{H}_i$  satisfies the conditions \eqref{3.1} and \eqref{3.2}.

There exists a transformation operator $e^{-\varepsilon \mathrm{i} \mathcal{B}_i(\theta,\omega)} $ conjugating $\mathcal{H}_i$ to $\mathcal{H}_{i+1}$, where
\begin{align}
\mathcal{H}_{i+1}=&\mathcal{K} +\mathcal{Q}+\varepsilon\mathcal{Z}^i+\varepsilon\langle \mathcal{W}^i\rangle \\
\label{4.2}
&+\varepsilon \big( -\omega \cdot \partial_{\theta}+\mathrm{i}[\mathcal{B}_i,\mathcal{K}]+\mathcal{W}^i-\langle \mathcal{W}^i\rangle \big) \\
\label{4.30}
&+e^{\varepsilon \mathrm{i} \mathcal{B}_i(\theta,\omega)} \mathcal{K} e^{-\varepsilon \mathrm{i} \mathcal{B}_i(\theta,\omega)}-\mathcal{K}-\varepsilon \mathrm{i}[\mathcal{B}_i,\mathcal{K}]\\
\label{4.4}
&+\varepsilon e^{\varepsilon \mathrm{i} \mathcal{B}_j(\theta,\omega)} \mathcal{Z}^{i} e^{-\varepsilon \mathrm{i} \mathcal{B}_i(\theta,\omega)}-\varepsilon \mathcal{Z}^{i}\\
\label{4.5}
&+e^{\varepsilon \mathrm{i} \mathcal{B}_j(\theta,\omega)} \mathcal{Q} e^{-\varepsilon \mathrm{i} \mathcal{B}_i(\theta,\omega)}-\mathcal{Q}\\
\label{4.7}
&+\varepsilon e^{\varepsilon \mathrm{i} \mathcal{B}_i(\theta,\omega)} \mathcal{W}^{i} e^{-\varepsilon \mathrm{i} B_i(\theta,\omega)}-\varepsilon \mathcal{W}^{i}\\
\label{4.6}
&+\mathrm{i}\varepsilon^2 \int^1_0(1-s)e^{\varepsilon \mathrm{i} s \mathcal{B}_i(\theta,\omega)}\mathrm{i}[\mathcal{B}_i,\omega\cdot \partial_{\theta} \mathcal{B}_i]e^{-\varepsilon \mathrm{i} s \mathcal{B}_i(\theta,\omega)}ds.
\end{align}

From Lemma \ref{4.3}, there exists an operator  $\mathcal{B}_{i}$  making   \eqref{4.2} equals to zero . From Remark \ref{8.2} and Lemma \ref{8.3}, we have
\begin{equation*}
\eqref{4.30} \in \mathcal{L}ip (\Omega_{0,\alpha},C^{\infty}(\mathbb{T}^d, OPS^{-i})),
\end{equation*}
\begin{equation*}
\eqref{4.4} \in \mathcal{L}ip (\Omega_{0,\alpha},C^{\infty}(\mathbb{T}^d, OPS^{\frac{1}{2}-\frac{1}{2}(i+1)})),
\end{equation*}
\begin{equation*}
\eqref{4.5} \in \mathcal{L}ip (\Omega_{0,\alpha},C^{\infty}(\mathbb{T}^d, OPS^{-\frac{3}{2}-\frac{1}{2}i})),
\end{equation*}
\begin{equation*}
\eqref{4.7} \in \mathcal{L}ip (\Omega_{0,\alpha},C^{\infty}(\mathbb{T}^d, OPS^{-i})),
\end{equation*}
\begin{equation*}
\eqref{4.6} \in \mathcal{L}ip (\Omega_{0,\alpha},C^{\infty}(\mathbb{T}^d, OPS^{-i})).
\end{equation*}
Rearranging the expression of $\mathcal{H}_{i+1}$ and setting
 $$ \varepsilon \mathcal{Z}^{i+1}=\varepsilon \mathcal{Z}^i+\varepsilon \langle \mathcal{W}^i \rangle,$$
 $$\varepsilon \mathcal{W}^{j+1}=\eqref{4.30}+\eqref{4.4}+\eqref{4.5}+\eqref{4.7}+\eqref{4.6}.$$
 Now, $\mathcal{Z}^{i+1}$ and $\mathcal{W}^{i+1}$ satisfy the hypothesis \eqref{3.1} and \eqref{3.2} with $i+1$.
It is easy to verified that $\mathcal{Z}^{i+1}$ and $\mathcal{W}^{i+1}$ are Hermitian operators.
\end{proof}
\begin{rem}
From Lemma \ref{8.1}, for all $j=0,1,2, \cdots, M$, the operator $e^{\pm\mathrm{i}\varepsilon \mathcal{B}_{j}} \in \mathcal{L}(\mathrm{H}^{r}),\forall r\geq 0$, and
\begin{equation}
\|e^{\pm \mathrm{i} \varepsilon \mathcal{B}_{j}}-\mathrm{Id}\|_{\mathcal{L}(\mathrm{H}^r,\mathrm{H}^{r-(\frac{1}{2}-\frac{1}{2}j)})}\lesssim \varepsilon \| \mathcal{B}_{j}\|_{\mathcal{L}(\mathrm{H}^r,\mathrm{H}^{r-(\frac{1}{2}-\frac{1}{2}j)})} .
\end{equation}
\end{rem}
Moreover, we also show that the closed set $\Omega_{0,\alpha}$ is asymptotically full Lebesgue.
\begin{prop}\label{4.8}
$$meas(\Omega \backslash \Omega_{0,\alpha}) \leq C \alpha.$$
\end{prop}
\begin{proof}
Set $Q_{\ell,\mathrm{m}}$ as
\begin{equation}
\Big\{\omega \in \Omega:|\omega \cdot \ell+\mathrm{m}|<\frac{\alpha}{1+|\ell|^{d+2}} \Big\}.
\end{equation}

If $|\ell| <\frac{ |\mathrm{m}|}{2}$, the set $Q_{\ell,m}$ is empty.

If $|\ell| \geq \frac{|\mathrm{m}|}{2} $,  one gets that
\begin{equation}
meas(Q_{\ell,\mathrm{m}}) \leq \frac{2\alpha}{1+|\ell|^{d+2}}.
\end{equation}
Finally, we have
\begin{equation}
meas(\Omega\backslash \Omega_{0,\alpha}) \leq meas (\bigcup_{(\ell,\mathrm{m})\in \Z^{d+1}\backslash \{0\}}Q_{\ell,m})\leq \sum_{|\mathrm{m}| \leq 2|\ell|, \ell \in \mathbb{Z}^d} meas(Q_{\ell,\mathrm{m}}) \leq C\alpha.
\end{equation}

\end{proof}

\section{KAM reducibility}

\subsection{The reducibility theorem}\

In this paper, the number of regularization step in Theorem \ref{3.0} is
\begin{equation}
M:=4m+1.
\end{equation}
After $M$ steps of regularization in the previous section, we get the new equation
\begin{equation}\label{R1}
\mathrm{i}\omega \cdot \partial_{\theta}u=\mathcal{H}^{M}u=\mathbf{\Lambda}^0u+ \mathbf{P}^0u,
\end{equation}
where $\mathbf{\Lambda}^0=\mathcal{K}+\mathcal{Q}+\varepsilon \mathcal{Z}^M$ and  $\mathbf{P}^0 =\varepsilon \mathcal{W}^M$.

The equation \eqref{R1} satisfies the following assumptions:\

\textbf{(A1)} The linear operator $\mathbf{\Lambda}^0$ is an Hermitian operator, block diagonal, and independent of $\theta$, Lipschitz on $\omega \in\Omega_{0,\alpha}$. Denoting $(\lambda_{j,k})_{k=1,2}$ as the eigenvalue of the  block $(\mathbf{\Lambda}^0)^{[j]}_{[j]}$, for any $\omega \in \Omega_{0,\alpha}$, there exists a constant $c_0$ such that
\begin{equation}
|\lambda_{i,k}-\lambda_{j,k^{'}}| \geq c_0|i-j|, \quad \forall k,k'=1,2, \ and \ i \neq j,
\end{equation}
\begin{equation}
|\lambda_{j,k}(\omega)|^{lip,\Omega_{0,\alpha}}=\sup_{\omega_1,\omega_2 \in \Omega_{0,\alpha}}\frac{|\lambda_{j,k}(\omega_1)-\lambda_{j,k}(\omega_2)|}{|\omega_1-\omega_2|} \leq \frac{1}{8},  \quad \forall j \in \mathbb{N}, \  k=1,2.
\end{equation}

\textbf{(A2)} The linear operator $\mathbf{P}^0$  is an Hermitian operator and belongs to  $\mathcal{M}^{S,\mathcal{L}ip}_{S-m,S+m}$, $S >\frac{d+1}{2}$.
\begin{rem}
The assumption \textbf{(A2)} can be obtained  from the Theorem \ref{3.0} and Prop \ref{8.5}. For assumption \textbf{(A1)}, we need the following Lemma.
\end{rem}

\begin{lem}\label{z1}
Suppose that $\mathcal{W}(\omega t) \in C^{\infty}(\mathbb{T}^d, OPS^{\frac{1}{2}})$, the eigenvalues $(\lambda_{j,k})_{k=1,2}$  of the  block $(\mathbf{\Lambda}_0)^{[j]}_{[j]}$ have the asymptotic expression
\begin{equation}
\lambda_{j,k}=(j^2+\mathfrak{m}^2)^{\frac{1}{2}}+\varepsilon a\langle j\rangle^{\frac{1}{2}}+r_{j,k}.
\end{equation}
where $ |r_{j,k}|^{\mathcal{L}ip,\Omega_{0,\alpha}} \leq C \varepsilon.$
\end{lem}
\begin{proof}
From Theorem \ref{3.0}, one gets
\begin{equation*}
\mathbf{\Lambda}^0=\mathcal{K}+\mathcal{Q}+\varepsilon \mathcal{Z}^M,
\end{equation*}
where $\mathcal{Z}^M=\langle \mathcal{W}^0\rangle+\langle \mathcal{W}^1\rangle+ \cdots +\langle \mathcal{W}^{M-1}\rangle$. We know that $\langle \mathcal{W}^0\rangle=\langle \mathcal{W}\rangle \in OPS^{\frac{1}{2}}$, and $\langle \mathcal{W}^1\rangle+\langle \mathcal{W}^2\rangle+ \cdots +\langle \mathcal{W}^{M-1}\rangle \in OPS^{0}$. The symbol of pseudo-differential operator  $\mathcal{W}$ can be written as
\begin{equation}
w(\theta,x,j)=\sum_{\ell \in \Z^{d}}w(x,j)(\ell)e^{\mathrm{i}\ell \cdot \theta}= \sum_{(\ell,k) \in \Z^{d+1}}w_{\ell,k}(j)e^{\mathrm{i} \ell \cdot \theta}e^{\mathrm{i}k \cdot x}.
\end{equation}
From Definition \ref{defn:symbols20} and \eqref{3.11}, \eqref{4.07}, one has
\begin{equation}
\langle\mathcal{W(\theta)}\rangle^{[j]}_{[j]}=\begin{pmatrix} w_{0,0}(j) & w_{0,-2j}(j)\\
w_{0,2j}(-j) & w_{0,0}(-j)
  \end{pmatrix}.
\end{equation}
These four elements in the matrix are independent of $\omega$. From Definition \ref{defn:symbols_semi} and condition \textbf{(C2)},  one gets
\begin{equation*}
w_{0,0}(\pm j)=\frac{1}{(2\pi)^{d+1}}\int_{\mathbb{T}^{d+1}}w(\theta,x,\pm j)dxd\theta=a\langle j \rangle^{\frac{1}{2}}+b(j),
\end{equation*}
\begin{equation*}
|w_{0,\pm2j}(\pm j)| \leq C \frac{\sum_{\beta \leq 1}\sup_{x \in \mathbb{T}}|\partial_{x}^{\beta}w(x,\pm j)(0)|}{1+|2j|} \leq \tilde{C} \frac{\langle j \rangle^{\frac{1}{2}}}{1+|2j|}.
\end{equation*}

We write  $\langle \mathcal{W}\rangle $ as $\langle \mathcal{W}\rangle_{A}+\langle \mathcal{W}\rangle_{B}$,
where
\begin{equation}
[\langle\mathcal{W(\theta)}\rangle_A]^{[j]}_{[j]}=\begin{pmatrix} a\langle j \rangle^{\frac{1}{2}} & 0 \\
0 & a\langle j \rangle^{\frac{1}{2}}
  \end{pmatrix},
  \quad
  [\langle\mathcal{W(\theta)}\rangle_B]^{[j]}_{[j]}=\begin{pmatrix} b(j) & w_{0,-2j}( j)\\
w_{0,2j}(- j) & b(-j)
  \end{pmatrix}.
\end{equation}

Denoting   $(\mu_{j,k})_{k=1,2}$ as the eigenvalues of the block $[\mathcal{K}+\mathcal{Q}+\langle\mathcal{W(\theta)}\rangle_A]^{[j]}_{[j]}$, one has $$\mu_{j,k}=(j^2+\mathfrak{m}^2)^{\frac{1}{2}}+a\langle j\rangle^{\frac{1}{2}}.$$

Let $\mathcal{R}=\langle\mathcal{W(\theta)}\rangle_B +\langle \mathcal{W}^1\rangle+ \cdots +\langle \mathcal{W}^{M-1}\rangle$, from  Theorem \ref{3.0} and Prop \ref{8.5}, for any $S>\frac{d+1}{2}$, one has $\mathcal{R} \in \mathcal{M}^{S,\mathcal{L}ip}_{S,S} $. From Prop \ref{8.6} and Corollary A.7 in \cite{F.G2019}, the Lipschitz variation of the eigenvalues of an Hermitian matrix is controlled by the Lipschitz variation of the matrix. Then, we can get
\begin{equation}
|r_{j,k}|^{\mathcal{L}ip}=|\lambda_{j,k}-\mu_{j,k}|^{\mathcal{L}ip} \leq \|\varepsilon \mathcal{R}^{[j]}_{[j]}\|^{\mathcal{L}ip} \leq C\varepsilon.
\end{equation}

Finally, the Lemma is proved.
\end{proof}

$$Set  \quad \epsilon_0= \|\mathbf{P}^0\|^{S,\mathcal{L}ip}_{S-m,S+m}=\|\varepsilon \mathcal{W}^M\|^{S,\mathcal{L}ip}_{S-m,S+m}.$$
The main goal of this section is the following  theorem.
\begin{thm}\label{5.30}\textbf{(The Reducibility Theorem)}Let $s \in[s_0,S-\beta]$, $r\in[0,S-\beta-\frac{d+1}{2})$ and  $\alpha \in (0,1)$. There exists a positive  $\epsilon_0:=\epsilon_0(s,d)$ such that, if  $\epsilon \leq \epsilon_0$, there exists a Cantor subset $\Omega_{\epsilon} \subseteq \Omega_{0,\alpha}$ with $$meas(\Omega_{0,\alpha} \backslash \Omega_{\epsilon}) \leq C \alpha.$$ For any $\omega \in \Omega_{\epsilon} $, there exist a family of bounded and invertible operators $\Phi_{\infty}:=\Phi_{\infty}(\omega,\theta) \in \mathcal{L}(\mathrm{H}^s)$ conjugating the linear equation \eqref{R1} to
\begin{equation}
\mathrm{i} \omega \cdot \partial_{\theta}u=\mathbf{H}^{\infty}u,
\end{equation}
  where $\mathbf{H}^{\infty}$ is a time independent and  block-diagonal Hamiltonian operator. Moreover, we have
  \begin{equation}
  \sup_{\theta \in \mathbb{T}^d}\|\Phi_{\infty}^{\pm1}(\theta)-\mathrm{Id}\|_{\mathcal{L}(\mathrm{H}^r)} \leq C\epsilon_0, \quad \forall \omega \in \Omega_{\epsilon}.
  \end{equation}
\end{thm}
The procedure of KAM iteration  is well known. For the convenience  of  reader, we show an outline of one step of the KAM reducibility.

Here, we conjugate the linear equation
$$\mathrm{i}\partial_{t}u=\mathbf{H}(t)u=\mathbf{\Lambda} u+\mathbf{P}(t) u$$
through a transformation $u=e^{-\mathrm{i}\mathbf{G}}v$, so that the new equation is
\begin{equation}
\mathrm{i}\partial_{t}v=\mathbf{H}^+(t)v,
\end{equation}
where
\begin{align}
		\label{eq:reducibility_1}
		\mathbf{H}^+(t) & = e^{\mathrm{i}\mathbf{G}(\omega t,\omega)}\mathbf{H}(t)e^{-\mathrm{i}\mathbf{G}(\omega t,\omega)}-\int^1_0 e^{\mathbf{i}s\mathbf{G}(\omega t,\omega)}\dot{\mathbf{G}}e^{-\mathbf{i}s\mathbf{G}(\omega t,\omega)}ds, \\
		\label{eq:reducibility 2}
		& \mathbf{H}^+= \mathbf{\Lambda} +\mathbf{i}[\mathbf{G},\mathbf{\Lambda}]+\Pi_{N}\mathbf{P} - \dot{\mathbf{G}} +\mathbf{P}^{+},
	\end{align}
and
\begin{align}
		\label{eq:reducibility_3}
\mathbf{P}^{+} =& e^{\mathrm{i}\mathbf{G}(\omega t,\omega)}\mathbf{\Lambda} e^{-\mathrm{i}\mathbf{G}(\omega t,\omega)}-(\mathbf{\Lambda}+\mathrm{i}[\mathbf{G},\mathbf{\Lambda}])+(e^{\mathrm{i}\mathbf{G}(\omega t,\omega)}\mathbf{P}
e^{-\mathrm{i}\mathbf{G}(\omega t,\omega)}-\mathbf{P}) \\
&-(\int^1_0e^{\mathrm{i}s\mathbf{G}(\omega t,\omega)}\dot{\mathbf{G}}
e^{-\mathrm{i}s\mathbf{G}(\omega t,\omega)}ds-\dot{\mathbf{G}}  )+\Pi^{\perp}_{N}\mathbf{P}.
\end{align}
Our goal is determine the operator $\mathbf{G}$ by solving the homological equation
\begin{equation}\label{1.2}
\omega \cdot \partial_{\theta}\mathbf{G}=\mathrm{i}[\mathbf{G},\mathbf{\Lambda}]+\Pi_{N}\mathbf{P}-\mathrm{diag}\big\{[\mathbf{P}^{[j]}_{[j]}](\omega)| j\in \mathbb{N}\big\}.
\end{equation}
Here $[\mathbf{P}^{[j]}_{[j]}](\omega)$ denotes
\begin{equation}
[\mathbf{P}^{[j]}_{[j]}](\omega)=\int_{\mathbb{T}^d}\mathbf{P}^{[j]}_{[j]}(\theta,\omega) d\theta.
\end{equation}

The new Hamiltonian is
\begin{align}
\mathbf{H}^{+}(t)=\mathbf{\Lambda}^{+}+\mathbf{P}^{+},\quad \mathbf{\Lambda}:=\mathrm{diag}\Big\{\Lambda_j(\omega)\Big|j\in \mathbb{N}\Big\}, \quad \mathbf{\Lambda}^{+}=\mathbf{\Lambda}+\mathrm{diag}\Big\{[\mathbf{P}^{[j]}_{[j]}](\omega)\Big| j\in \mathbb{N}\Big\}.
\end{align}

It is well known that the crucial of KAM iteration is to estimate the solution $\mathbf{G}$ of homological equation \eqref{1.2} . In order to deal with the  notorious small divisor, some non-resonance conditions on the eigenvalues of diagonal operator $\mathbf{\Lambda}$ are necessary.

Denoting $(\lambda^k_{j,v})_{v=1,2}$ as the eigenvalues of the  block $\Lambda_j$, we define the non-resonance  set $\Omega_{k+1,\alpha}(\omega)$ at the $k+1^{th}$ step KAM reducibility as
\begin{multline} \label{5.21}
\Omega_{k+1,\alpha}:=\Big\{\omega \in \Omega_{k,\alpha}:|\omega \cdot \ell+\lambda^{k}_{i,v}-\lambda^{k}_{j,v'}| \geq  \frac{\alpha}{N_k^{\tau}\langle i \rangle^{\sigma} \langle j \rangle^{\sigma}}, \\
\forall i,j\in\mathbb{ N}, \quad |\ell| \leq N_k, \quad v, v'=1,2, \quad (\ell,i,j) \neq (0,i,i)
\Big\}.
\end{multline}

In the following section, we will  estimate the solution $\mathbf{G}^{k+1}$ of homological equation \eqref{1.2} and the new perturbation $\mathbf{P}^{k+1}$ in the KAM procedure.
\subsection{The homological equation}
\begin{lem}\label{5.111}For any $\omega \in \Omega_{k+1,\alpha} $ and  $s\in[s_0, S-\beta]$, the homological equation
\begin{equation}\label{4.1}
\omega \cdot \partial_{\theta}  \mathbf{G}^{k+1}+\mathrm{i}[\mathbf{\Lambda}^k,\mathbf{G}^{k+1}]=\Pi_{N^k}\mathbf{P}^k- \mathrm{diag}[\mathbf{P}^k]
\end{equation}
has a solution $\mathbf{G}^{k+1}$ defined on  $\Omega_{k+1,\alpha}$ with
\begin{equation}\label{5.3}
\|\mathbf{G}^{k+1}\|^{s,\mathcal{L}ip}_{s\mp m,s\mp m} \lesssim N^{2\tau+2\sigma+2}_{k} \|\mathbf{P}^k\|^{s,\mathcal{L}ip}_{s- m,s+ m},
\end{equation}
\begin{equation}\label{5.4}
\|\mathbf{G}^{k+1}\|^{s+\beta,\mathcal{L}ip}_{s+\beta\mp m,s+\beta\mp m} \lesssim N^{2\tau+2\sigma+2}_{k} \|\mathbf{P}^k\|^{s+\beta,\mathcal{L}ip}_{s+\beta- m,s+\beta+ m}.
\end{equation}
\end{lem}
\begin{proof}
For notation simplicity, we rename $\mathbf{\Lambda}^k,\mathbf{G}^{k+1},\mathbf{P}^k,\lambda^{k}_i,N_k$ as $\mathbf{\Lambda},\mathbf{G},\mathbf{P},\lambda_i,N$. Considering the matrix representation  and Fourier coefficients of these linear operators, the homological equation \eqref{4.1} is equivalent to
\begin{equation}\label{5.2}
\mathrm{i}\omega \cdot \ell \hat{\mathbf{G}}^{[i]}_{[j]}(\ell)+\mathrm{i}\Lambda_i \hat{\mathbf{G}}^{[i]}_{[j]}(\ell)- \mathrm{i} \hat{\mathbf{G}}^{[i]}_{[j]}(\ell)\Lambda_j=\hat{\mathbf{P}}^{[i]}_{[j]}(\ell), \quad  \forall  |i-j|<N, \ |\ell|<N, \ (\ell,i,j)\neq (0,i,i).
\end{equation}
and  $\hat{\mathbf{G}}^{[i]}_{[i]}(0)=0$.

From \eqref{5.21}  and Prop \ref{8.7}, for any $|i-j|<N$, one has
\begin{equation}
\begin{split}
\|\hat{\mathbf{G}}^{[i]}_{[j]}(\ell)\| &\lesssim \frac{\|\hat{\mathbf{P}}^{[i]}_{[j]}(\ell)\|N^{\tau} \langle i\rangle^{\sigma}\langle j\rangle^{\sigma}}{\alpha}\\
&\lesssim \alpha^{-1} \|\hat{\mathbf{P}}^{[i]}_{[j]}(\ell)\| N^{\tau} \langle j\rangle^{\sigma}(\langle j\rangle^{\sigma}+|i-j|^{\sigma})\\
&\lesssim  \alpha^{-1} \|\hat{\mathbf{P}}^{[i]}_{[j]}(\ell)\|N^{\tau}\langle j\rangle^{\sigma}(\langle j\rangle^{\sigma}+N^{\sigma})\\
&\lesssim  \alpha^{-1} \|\hat{\mathbf{P}}^{[i]}_{[j]}(\ell)\|N^{\tau+\sigma}\langle j\rangle^{2\sigma}.
\end{split}
\end{equation}

From the definition of the norm $\|\cdot\|^{s}_{s+ m,s+ m}$, we can get
\begin{align}
\big(\|\mathbf{G}\|^{s}_{s+ m,s+ m}\big)^2 &= \sum_{\ell \in \Z^d,  h \in \mathbb{N }}\langle \ell,h  \rangle^{2s}   \sup_{|i-j|=h} \|\mathbf{G}^{[i]}_{[j]}(\ell)\|^2\langle j\rangle^{-2m} \langle i \rangle^{2m} \\
&\lesssim \alpha^{-2}N^{2\sigma+2\tau} \sum_{\ell \in \Z^d,  h \in \mathbb{N }}\langle \ell,h  \rangle^{2s} \sup_{|i-j|=h} \|\hat{\mathbf{P}}^{[i]}_{[j]}(\ell)\|^2\langle i\rangle^{2m} \langle j\rangle^{4\sigma} \langle j \rangle^{-2m} \\
&\label{5.33}\lesssim \alpha^{-2}N^{2\sigma+2\tau} \sum_{\ell \in \Z^d,  h \in \mathbb{N }}\langle \ell,h  \rangle^{2s} \sup_{|i-j|=h} \|\hat{\mathbf{P}}^{[i]}_{[j]}(\ell)\|^2 \langle i\rangle ^{2m} \langle j \rangle^{2m} \\
&\lesssim \alpha^{-2}N^{2\sigma+2\tau}\big(\|\mathbf{P}\|^{s}_{s- m,s+ m}\big)^2
\end{align}
The  inequality \eqref{5.33} is valid, because $\sigma \leq m$ and  $4\sigma -2m \leq 2m$.  By the same way, we also have
$$\big(\|\mathbf{G}\|^{s}_{s- m,s-m}\big)^2 \lesssim \alpha^{-2}N^{2\sigma+2\tau}\big(\|\mathbf{P}\|^{s}_{s- m,s+ m}\big)^2. $$
There is no difference in estimating  $\|\mathbf{G}\|^{s+\beta}_{s+\beta\mp m,s+\beta\mp m}$ with $\|\mathbf{G}\|^{s}_{s\mp m,s\mp m}.$

Regarding the Lipschitz semi-norm of $\mathbf{G}$, we introduce the difference operator $\Delta$. Given the operator $\mathbf{G}$ of $\omega$, we set $\Delta\mathbf{G}=\mathbf{G}(\omega_1)-\mathbf{G}(\omega_2)$. Applying the difference operator $\Delta$ to equation  \eqref{5.2}, we have
\begin{equation}
\begin{split}
\mathrm{i}\omega \cdot \ell (\Delta \hat{\mathbf{G}}^{[i]}_{[j]}(\ell))+\mathrm{i}\Lambda_i (\Delta\hat{\mathbf{G}}^{[i]}_{[j]}(\ell))- \mathrm{i} (\Delta\hat{\mathbf{G}}^{[i]}_{[j]}(\ell)) \Lambda_j=&\Delta(\hat{\mathbf{P}}^{[i]}_{[j]}(\ell))-\mathrm{i}\Delta \omega \cdot \ell \hat{\mathbf{G}}^{[i]}_{[j]}(\ell)\\
&-\mathrm{i}(\Delta\Lambda_i)\hat{\mathbf{G}}^{[i]}_{[j]}(\ell)+ \mathrm{i}\hat{\mathbf{G}}^{[i]}_{[j]}(\ell)(\Delta \Lambda_j).
\end{split}
\end{equation}
Applying Prop \ref{8.7} again, we have
\begin{align}
\frac{\|\Delta  \hat{\mathbf{G}}^{[i]}_{[j]}(\ell)\|}{|\Delta \omega|} \lesssim &\frac{N^{\tau} \langle i\rangle^{\sigma}  \langle j\rangle^{\sigma} }{\alpha} \Big( \frac{\|\Delta \hat{\mathbf{P}}^{[i]}_{[j]}(\ell)\|}{|\Delta \omega|}
+\|\hat{\mathbf{G}}^{[i]}_{[j]}(\ell)\|\langle \ell \rangle+\|\hat{\mathbf{G}}^{[i]}_{[j]}(\ell)\|\langle i+j\rangle \Big)\\
\lesssim & \frac{N^{\tau} \langle i\rangle^{\sigma}  \langle j\rangle^{\sigma} }{\alpha}\frac{\|\Delta \hat{\mathbf{P}}^{[i]}_{[j]}(\ell)\|}{|\Delta \omega|}+\frac{N^{2\tau+1}\langle i\rangle^{2\sigma+1}  \langle j\rangle^{2\sigma+1}}{\alpha^2}\| \hat{\mathbf{P}}^{[i]}_{[j]}(\ell)\|.
\end{align}
Now, we can get

\begin{equation}\label{5.5}
\frac{\|\Delta  \mathbf{G}\|^{s}_{s\mp m,s\mp m}}{|\Delta \omega|} \lesssim \frac{N^{\tau+\sigma}}{\alpha} \frac{\|\Delta\mathbf{P}\|^{s}_{s- m,s+ m}}{|\Delta \omega|}+\frac{N^{2\tau+2\sigma+2}}{\alpha^2}\|\mathbf{P}\|^{s}_{s- m,s+ m}.
\end{equation}
It is same to consider  $\frac{\|\Delta  \mathbf{G}\|^{s+\beta}_{s+\beta\mp m,s+\beta\mp m}}{|\Delta  \omega |}$.
Respectively, we can get \eqref{5.3} and \eqref{5.4}.

\end{proof}
Next, we consider the new perturbation $\mathbf{P}^{k+1}$.
\begin{lem}\label{5.112}
Assuming that $C(s)\|\mathbf{P}^k\|^{s,\mathcal{L}ip}_{s-m,s+ m}\leq \frac{1}{2}$, the new perturbation $\mathbf{P}^{k+1}$ is defined on $\Omega_{k+1,\alpha}$, and satisfies the following quantities bounds:
\begin{equation}\label{6.1}
\|\mathbf{P}^{k+1}\|^{s,\mathcal{L}ip}_{s-m,s+ m} \leq C\Big(N^{-\beta}_k\|\mathbf{P}^k\|^{s+\beta,\mathcal{L}ip}_{s+\beta-m,s+\beta +m}+N^{2\tau+2\sigma+2}_k  (\|\mathbf{P}^k\|^{s,\mathcal{L}ip}_{s-m,s+ m})^2\Big),
\end{equation}
\begin{equation}\label{6.2}
\|\mathbf{P}^{k+1}\|^{s+\beta,\mathcal{L}ip}_{s+\beta-m,s+\beta+ m} \leq C\Big( \|\mathbf{P}^{k}\|^{s+\beta,\mathcal{L}ip}_{s+\beta-m,s+\beta+ m} +N^{2\tau+2\sigma+2}_k  \|\mathbf{P}^k\|^{s,\mathcal{L}ip}_{s-m,s+ m}\|\mathbf{P}^k\|^{s+\beta,\mathcal{L}ip}_{s+\beta-m,s+\beta+ m}\Big).
\end{equation}
Here, $C$ is a constant depending on $s,m,\sigma,\tau$.
\end{lem}
\begin{proof}
Recall the definition  of $\mathbf{P}^{k+1}$, we have
\begin{align}
\mathbf{P}^{k+1}=\Pi^{\perp}_{N_k}\mathbf{P}^{k}&+\int^{1}_{0}e^{\mathbf{i}s\mathbf{G}^{k+1}}\mathrm{i}[\mathbf{G}^{k+1},\mathbf{P}^{k}]e^{-\mathbf{i}s\mathbf{G}^{k+1}}ds\\
&+\int^{1}_{0} (1-s)e^{\mathbf{i}s\mathbf{G}^{k+1}}\mathrm{i}[\mathbf{G}^{k+1},[\mathbf{P}^{k}]-\Pi_{N_k}\mathbf{P}^{k}]e^{-\mathbf{i}s\mathbf{G}^{k+1}}ds.
\end{align}
The Lemma \ref{2.1} implies that
\begin{equation}
\|\Pi^{\perp}_{N_k}\mathbf{P}^{k}\|^{s,\mathcal{L}ip}_{s-m,s+m} \leq N_{k}^{-\beta}\|\mathbf{P}^{k}\|^{s+\beta,\mathcal{L}ip}_{s+\beta-m,s+\beta+ m}.
\end{equation}
From Lemma \ref{3.3} and \eqref{5.3}, we can get
\begin{equation}
\|[\mathbf{G}^{k+1},\mathbf{P}^k]\|^{s,\mathcal{L}ip}_{s-m,s+m} \lesssim N^{2\tau+2\sigma+2}_k\Big(  \|\mathbf{P}^{k}\|^{s,\mathcal{L}ip}_{s-m,s+m} \Big)^2,
\end{equation}
and
\begin{equation}
\|[\mathbf{G}^{k+1},\mathbf{P}^k]\|^{s+\beta,\mathcal{L}ip}_{s+\beta-m,s+\beta+m} \lesssim N^{2\tau+2\sigma+2}_k \|\mathbf{P}^{k}\|^{s,\mathcal{L}ip}_{s-m,s+m} \|\mathbf{P}^{k}\|^{s+\beta,\mathcal{L}ip}_{s+\beta-m,s+\beta+m}.
\end{equation}

The estimation of $[\mathbf{G}^{k+1},[\mathbf{P}^{k}]-\Pi_{N_k}\mathbf{P}^{k}]$ is same with $[\mathbf{G}^{k+1},\mathbf{P}^k]$.
Summing up the contribution of these operators and using Lemma \ref{2.4}, we can obtain  \eqref{6.1} and \eqref{6.2} respectively.
\end{proof}

\subsection{Proof of the reducibility theorem }\
\subsubsection{Iterative Lemma}\

The proof of the theorem \ref{5.30} is heavily  depending on the following iterative lemma. Some constants should be fixed before the following lemma. Given $\tau >d+1,\sigma>1$, we fix
\begin{equation}\label{5.45}
s_0=\frac{d+3}{2}, \quad m=2\sigma+2, \quad \alpha=6\tau+6\sigma+7, \quad \beta=\alpha+1.
\end{equation}
Moreover, we fix the scale on which we perform the reducibility scheme as
\begin{equation}\label{5.46}
N_{k}=(N_0)^{(\frac{3}{2})^k}, \quad \forall k \in \mathbb{N} , \quad N_{-1}=1.
\end{equation}

\begin{prop}\label{5.6}(Iterative Lemma)Let $s \in [s_0,S-\beta]$.  There exists  $C(s)>0$ and $N_0:= N_0(s)>1$ such that if
\begin{equation}
C(s)N_0^{2\tau+2\sigma+2+\alpha}\epsilon  \leq \frac{1}{2},
\end{equation}
 we can recursively define a family of non-resonance set $\{\Omega_{n}\}_{n \geq 0}$. For any $\omega \in \Omega_{n}$, we can iteratively define a Lipschitz family of linear operator
\begin{equation}
\mathcal{L}_{n}=\mathrm{i} \omega \cdot \partial_{\theta}-\mathbf{\Lambda}^n-\mathbf{P}^n, \quad n\geq 0,
\end{equation}
such that the followging items hold true for any $n\geq 0$:\\
\textbf{(A)}: For any $n\geq 1$, there exists  a Lipschitz family transformation operator $e^{-\mathrm{i}\mathbf{G}^{n}}$ defined on $\Omega_{n}$, which conjugate the linear operator $\mathcal{L}_{n-1}$ to
\begin{equation}
\mathcal{L}_{n}=e^{\mathrm{i}\mathbf{G}^{n}}\mathcal{L}_{n-1}e^{-\mathrm{i}\mathbf{G}^{n}}.
\end{equation}
Moreover, for any $s \in [s_0,S-\beta]$
\begin{equation}
\|\mathbf{G}^{n}\|^{s,\mathcal{L}ip}_{s\pm m,s\pm m} \leq C_{\star} N^{2\tau+2\sigma+2}_{n-1} N^{-\alpha}_{n-2}\epsilon.
\end{equation}
\textbf{(B)}:  $\mathbf{\Lambda}^{n}$ is   block diagonal and time independent. Denoting  $(\lambda^{n}_{j,v})_{v=1,2}$ as the  eigenvalues of block $\Lambda^{n}_j$, for any $\omega \in \Omega_{n}$, there exists a positive constant $c_0$ such that
\begin{equation}
|\lambda^{n}_{i,v}-\lambda^{n}_{j,v'}| \geq \frac{c_0}{2}|i-j|,\  \forall i \neq j, \ \ v,v'=1,2,
\end{equation}
and
\begin{equation}\label{5.51}
|\lambda^{n}_{j,v}|^{lip} \leq \|\Lambda^{n}_j\|^{lip} \leq \frac{1}{4}.
\end{equation}
\textbf{(C)}:For any $s \in [s_0, S-\beta]$, the perturbation $\mathbf{P}^{n}$ is defined on $\Omega_{n}$, and satisfies
\begin{equation}
\|\mathbf{P}^{n}\|^{s,\mathcal{L}ip}_{s-m,s+m} \leq C_{\ast}N_{n-1}^{-\alpha}\epsilon
\end{equation}
\begin{equation}
\|\mathbf{P}^{n}\|^{s+\beta,\mathcal{L}ip}_{s+\beta-m,s+\beta+m} \leq C_{\ast}N_{n-1}
\end{equation}
The constant $C_{\ast}$  depends on $m,\sigma,\tau, s,d$.
  \end{prop}
\begin{proof}
We prove this proposition by induction method.

From the assumptions  \textbf{(A1)} and \textbf{(A2)}, the conditions \textbf{(B)},\textbf{(C)} are valid for $n=0$. We assume that  conditions \textbf{(A)},\textbf{(B)},\textbf{(C)} hold true for $1\leq n \leq k$. Our goal is to prove that they also hold for $n=k+1$.

From Lemma \ref{5.111}, for any $s \in [s_0, S-\beta]$ and $\omega  \in \Omega_{k+1}$, we have
\begin{align}
\|\mathbf{G}^{k+1}\|^{s,\mathcal{L}ip}_{s\mp m,s\mp m} \lesssim & N^{2\tau+2\sigma+2}_{k} \|\mathbf{P}^{k}\|^{s,\mathcal{L}ip}_{s- m,s+ m}\\
\leq&C_{\star} N^{2\tau+2\sigma+2}_k N^{-\alpha}_{k-1}\epsilon.
\end{align}
Hence, the  condition \textbf{(A)}  holds true for $n=k+1$.

From Lemma \ref{5.112}, for any $s\in [s_0,S-\beta]$ and $\omega  \in \Omega_{k+1}$, one gets
\begin{align}
\|\mathbf{P}^{k+1}\|^{s,\mathcal{L}ip}_{s-m,s+ m} &\leq C\Big(N^{-\beta}_k\|\mathbf{P}^k\|^{s+\beta,\mathcal{L}ip}_{s+\beta-m,s+\beta+ m}+N^{2\tau+2\sigma+2}_k  (\|\mathbf{P}^k\|^{s,\mathcal{L}ip}_{s-m,s+ m})^2\Big)\\
&\leq C C_{\ast}N_{k-1}N_k^{-\beta} \epsilon +CC_{\ast}^2N^{-2\alpha}_{k-1}N^{2\tau+2\sigma+2}_k\epsilon^2\\
&\leq C_{\ast}N_{k}^{-\alpha}\epsilon,
\end{align}
provided
\begin{equation}
2CN^{\alpha-\beta}_kN_{k-1} \leq 1, \quad 2CC_{\ast}N^{-2\alpha}_{k-1}N_k^{\alpha+2\tau+2\sigma+2}\epsilon  \leq 1.
\end{equation}
 These conditions can be verified by \eqref{5.45} and \eqref{5.46}. Furthermore, we have
\begin{align}
\|\mathbf{P}^{k+1}\|^{s+\beta,\mathcal{L}ip}_{s+\beta-m,s+\beta+ m} & \leq C\Big( \|\mathbf{P}^{k}\|^{s+\beta,\mathcal{L}ip}_{s+\beta-m,s+\beta+ m} +N^{2\tau+2\sigma+2}_k  \|\mathbf{P}^{k}\|^{s,\mathcal{L}ip}_{s-m,s+ m}\|\mathbf{P}^{k}\|^{s+\beta,\mathcal{L}ip}_{s+\beta-m,s+\beta+ m}\Big)\\
& \leq C C_{\ast}N_{k-1}\epsilon+ CC_{\ast}N_{k}^{2\tau+2\sigma+2}N^{-\alpha}_{k-1} N_{k-1}\epsilon^2 \\
& \leq C_{\ast}N_{k}\epsilon,
\end{align}
provided $N_0$ is big enough.\\
Hence, the condition \textbf{(C)} is valid for $n=k+1$.

Regarding the new diagonal operator  $\mathbf{\Lambda}^{k+1}=\mathbf{\Lambda}^{k}+\mathrm{diag}[\mathbf{P}^{k}]$, from Prop \ref{8.7}, $\forall i\neq j$, one has
\begin{align*}
|\lambda^{k+1}_{i,v}-\lambda^{k+1}_{j,v'}| &\geq |\lambda_{i,v}-\lambda_{j,v'}|-\big(\sum^{k}_{n=0}[(\mathbf{P}^n)^{[i]}_{[i]}]+\sum^{k}_{n=0}[(\mathbf{P}^n)^{[j]}_{[j]}] \big)\\
& \geq c_0|i-j|-2\sum^{k}_{n=0}\|\mathbf{P}^n\|^{s,\mathcal{L}ip}_{s-m.s+m} \\
& \geq c_0|i-j|-2C_{\star}\sum^{k}_{n=0}N^{-\alpha}_{n-1}\epsilon\\
&\geq \frac{c_0}{2}|i-j|.
\end{align*}

Since the Lipschitz variation of the eigenvalues of an Hermitian matrix is controlled by the Lipschitz variation of the matrix, one has
\begin{align*}
|\lambda^{k+1}_{j,v}|^{lip} & \leq \|\Lambda^{k+1}_j\|^{lip} \leq \frac{1}{8}+\sum^{k}_{n=0}\|(\mathbf{P}^n)^{[i]}_{[i]}\|^{lip}\\
& \leq \frac{1}{8}+C_{\star}\sum^{k}_{n=0}N^{-\alpha}_{n-1}\epsilon \leq \frac{1}{4}.
\end{align*}
Hence, the condition \textbf{(B)} is valid for $n=k+1$.
\end{proof}

Morover, we need  estimate the set of parameters excluded in the KAM iteration. Thus, we need the following assertions.

\subsubsection{Measure Estimates}\

In this section, we show that the set excluded in the KAM iteration  is  asymptotic  full measure. In the  iteration procedure, we have recursively defined the set $\{\Omega_{k,\alpha}\}, k \geq 0$, where $\Omega_{k+1,\alpha} \subseteq \Omega_{k,\alpha}, \ k \geq 0$.

Set $\Omega_{\infty, \alpha} =\bigcap^{\infty}_{i=0}\Omega_{i,\alpha}$, we prove the following assertion.
\begin{thm}\label{5.7}
\begin{equation}
meas(\Omega_{0,\alpha}\backslash \Omega_{\infty,\alpha}) \leq C \alpha.
\end{equation}
\end{thm}
Since $\Omega_{k+1} \subseteq \Omega_{k}, \ k \geq 0$, we can decompose $\Omega_{0,\alpha}\backslash \Omega_{\infty,\alpha}$ as
\begin{equation}
\Omega_{0,\alpha}\backslash \Omega_{\infty,\alpha}=\bigcup^{\infty}_{k=0}(\Omega_{k,\alpha}\backslash \Omega_{k+1,\alpha}).
\end{equation}
Obviously, to estimate  the measure  of  $(\Omega_{k,\alpha}\backslash \Omega_{k+1,\alpha})$ is crucial. From the definition of $\Omega_{k,\alpha}$, one has
\begin{equation*}
  \Omega_{k,\alpha}\backslash \Omega_{k+1,\alpha}\subseteq \bigcup_{\substack{\ell \in \Z^d ,|\ell| \leq N_k\\ |i-j|\leq N_k}} \bigcup_{\substack{(\ell,i,j) \neq (0,j,j)\\
 v,v'=1,2 }}R_{\ell ijvv'}
\end{equation*}
and
\begin{equation*}
R_{\ell ijvv'}=\Big\{\omega \in \Omega_{k,\alpha}: | \omega \cdot \ell  +\lambda^k_{i,v}-\lambda^k_{j,v'}|< \frac{\alpha}{N^{\tau}_k \langle i\rangle^{\sigma} \langle j \rangle^{\sigma}}\Big\}.
\end{equation*}

\begin{lem}\label{5.8}
\begin{equation*}
meas(\Omega_{k,\alpha}\backslash \Omega_{k+1,\alpha}) \leq C \alpha N^{-1}_{k}.
\end{equation*}
\end{lem}
\begin{proof}
If $\ell =0$ and $i \neq j$, we have
\begin{align}
|\lambda^k_{i,v}-\lambda^k_{j,v'}|\geq \frac{c_0}{2}|i-j| \geq \alpha.
\end{align}
Hence, $R_{0ijvv'}$ is an empty set.
For the other cases, we consider the Lipschitz function $g(\omega)$
$$g(\omega)=\omega \cdot \ell +\lambda^k_{i,v}(\omega)-\lambda^k_{j,v'}(\omega). $$
If $\ell \neq 0$, we write
\begin{equation}
\omega=\frac{\ell}{|\ell|}s+\omega_1, \quad \omega_1 \in \mathbb{R}^d, \quad \omega_1 \cdot \ell=0,
\end{equation}
and
$$g(s)=|\ell| \cdot s + \lambda^k_{i,v}(\omega(s))-\lambda^k_{j,v'}(\omega(s)).$$

From \eqref{5.51}, we can obtain
\begin{equation}
|g(s_1)-g(s_2)| \geq (|\ell|-\frac{1}{4})|s_1-s_2| \geq \frac{1}{2}|s_1-s_2|,
\end{equation}
which implies
\begin{equation}
meas \Big\{s\in \mathbb{R}: g_{\ell ijvv'}(s) <\frac{\alpha}{N^{\tau}_k \langle i\rangle^{\sigma} \langle j\rangle^{\sigma}}   \Big\} <\frac{2\alpha}{N^{\tau}_k \langle i\rangle^{\sigma} \langle j\rangle^{\sigma} }
\end{equation}

By the Fubini theorem, we can get

\begin{equation}
meas(R_{\ell i j v v'}) \leq \frac{2\alpha}{N^{\tau}_k \langle i\rangle^{\sigma} \langle j\rangle^{\sigma} }.
\end{equation}

Finally, we have
\begin{align*}
meas(\Omega_{k,\alpha}\backslash \Omega_{k+1,\alpha})\leq &\sum_{\substack{\ell \in \Z^d ,|\ell| \leq N_k\\ i,j \in \mathbb{N }}} \sum_{\substack{(\ell,i,j) \neq (0,j,j)\\
 v,v=1,2 }}meas(R_{\ell ijvv'})\\
 \leq & \sum_{\substack{\ell \in \Z^d ,|\ell| \leq N_k\\ i,j \in \mathbb{N }}} \frac{8 \alpha}{N^{\tau}_k \langle i\rangle^{\sigma} \langle j\rangle^{\sigma} }\\
 \leq & CN^{-1}_k \alpha
\end{align*}
\end{proof}
\begin{proof}(Proof of Theorem \ref{5.7}) From Lemma \ref{5.8}, provided $N_0$ is large enough, we have
\begin{equation}
meas(\Omega_{0,\alpha}\backslash \Omega_{\infty,\alpha}) \leq \sum^{\infty}_{k=0}  CN^{-1}_k \alpha \leq C\alpha.
\end{equation}
\end{proof}

From Proposition \ref{5.6} and  Theorem \ref{5.7}, we can give a proof the Theorem \ref{5.30}.
\begin{proof}(Proof of Theorem \ref{5.30}) For any $k\geq 0$, we can define a sequence linear operator
\begin{equation}
\Phi_k=e^{-\mathrm{i}\mathbf{G}_1}\circ e^{-\mathrm{i}\mathbf{G}_2} \circ \cdots e^{\mathrm{i}\mathbf{G}_k}
\end{equation}
on the set $\Omega_{\infty,\alpha}$.
The sequence of linear operator $\{\Phi_k\}_{k\geq1}$ is converges to an invertible operator $\Phi_{\infty}$, and satisfies
\begin{equation}
\|\Phi^{\pm}_{\infty}-\mathrm{Id}\|^{s,\mathcal{L}ip}_{s\pm m.s\pm m}\leq C(s) N^{2\tau+2\sigma+2}_0 \epsilon.
\end{equation}
From Lemma \ref{2.101} and Remark \ref{2.102}, for any $r \in[0 ,S-\beta-\frac{d+1}{2})$, there exists $ s \in[s_0,S-\beta]$ such that
\begin{equation}
\sup_{\theta \in \mathbb{ T}^d}\|\Phi^{\pm}_{\infty}-\mathrm{Id}\|_{\mathcal{L}(\mathrm{H}^{r})} \leq \|\Phi^{\pm}_{\infty}-\mathrm{Id}\|^{s}_{s\pm m.s\pm m}\leq C(s) N^{2\tau+2\sigma+2}_0 \epsilon.
\end{equation}
Passing the iterative Lemma \ref{5.6} to the limit, the operator $\mathcal{L}_0$  is conjugated to
$$\mathcal{L}_{\infty}=\mathrm{i} \omega \cdot \partial_{\theta}-\mathbf{\Lambda}^{\infty}$$
where $\mathbf{\Lambda}^{\infty}$ is a $\theta$ independent, block diagonal, Hermitian operator.
\end{proof}

\section{Proof of the  main result}
\begin{proof}(Proof of Theorem \ref{THM1}) We consider the composition operator
\begin{equation}
\mathcal{N}(\theta)= e^{-i\mathcal{B}_0(\theta)}\circ \cdots e^{-i\mathcal{B}_{M-1}(\theta)}\circ \Phi_{\infty}(\theta)
\end{equation}
defined on $\Omega_{\infty,\alpha}$.
From Theorem \ref{5.7} and Prop \ref{4.8}, one gets
\begin{equation}
meas(\Omega\backslash \Omega_{\infty,\alpha})\leq meas(\Omega\backslash \Omega_{0,\alpha})+meas(\Omega_{0,\alpha} \backslash \Omega_{\infty,\alpha}) \leq C\alpha.
\end{equation}
The coordinate transormation $u=\mathcal{N}(\theta)v $ transforms the equation \eqref{eq1} into
\begin{equation}
\mathrm{i} \partial_tv=\mathbf{\Lambda}^{\infty}v.
\end{equation}

From Lemma \ref{8.1} and Theorem \ref{5.7}, for any $r\geq 0$, there exists a finite constant $C_0$ such that
\begin{equation}\label{4.11}
\sup_{\theta \in \mathbb{T}^d}\|\mathcal{N}(\theta)\|_{\mathcal{L}(\mathrm{H}^{r})} \leq\sup_{\theta \in \mathbb{T}^d} \| e^{-i\mathcal{B}_0(\theta)}\circ \cdots e^{-i\mathcal{B}_{M-1}(\theta)}\|_{\mathcal{L}(\mathrm{H}^{r})}\sup_{\theta \in \mathbb{T}^d}\|\Phi_{\infty}(\theta)\|_{\mathcal{L}(\mathrm{H}^{r})} \leq C_0
\end{equation}
and
\begin{equation}\label{4.12}
\sup_{\theta \in \mathbb{T}^d}\|\mathcal{N}^{-1}(\theta)\|_{\mathcal{L}(\mathrm{H}^{r})} \leq \sup_{\theta \in \mathbb{T}^d}\|\Phi^{-1}_{\infty}(\theta)\|_{\mathcal{L}(\mathrm{H}^{r})}\sup_{\theta \in \mathbb{T}^d}\| e^{i\mathcal{B}_0(\theta)}\circ \cdots e^{i\mathcal{B}_{M-1}(\theta)}\|_{\mathcal{L}(\mathrm{H}^{r})} \leq C_0.
\end{equation}
Hence, the Theorem \ref{THM1} is proved.
\end{proof}

\section{Appendix A}
For the convenience of reader, we emphasize the difference between the proof of Theorem \ref{THM1} and of Theorem \ref{THM2}.

\textbf{The difference in functional space:}\

The Sobolev space $\mathrm{H}^r(\T_{\beta})$ is defined by
\begin{equation}
	\mathrm{H}^r(\T_{\beta}):=\Set{ u =\sum_{\xi \in \mathbb{Z}}\hat{u}(\xi)e^{\mathrm{i}\frac{x}{\beta}\cdot \xi}: |  \norm{u}_{\mathrm{H}^r(\T)}^2:= \sum_{\xi\in\Z}\braket{\xi}^{2r} \hat{u}(\xi)^2 <\infty  } .
\end{equation}
Similarly, we can define the pseudo-differential operator on the  irrational torus $\T_{\beta}$.
\begin{defn}\label{defn:symbols_}
	Given $m \in \mathbb{R}$, a function $a(x,\xi) \in C^{\infty}(\mathbb{T}_{\beta} \times \mathbb{Z})$ is called a symbol of class $S^m$ if for any
	$ \alpha, \beta \in \mathbb{N}$, there exists $C_{\alpha, \beta}>0$ such that
	$$
	\abs{\partial_x^\alpha \Delta^\beta a( x, \xi)} \leq C_{\alpha, \beta} \,  \langle \xi \rangle^{m - \beta}  \ , \quad \forall (x,\xi) \in \T_{\beta}\times \mathbb{Z}   \ .
	$$
\end{defn}

\begin{defn}\label{defn:symbols2}
	Given a symbol $a \in S^m$, we say that $Op(a)$ is the associated pseudo-differential operator of $a$ if for any $u \in L^2(\T_{\beta})$
\begin{equation}
Op(a)[u](x)=\sum_{\xi \in \Z} a(x, \xi) \hat{u}(\xi) e^{\mathrm{i} \frac{x}{\beta} \cdot \xi}.
\end{equation}
\end{defn}
Since  the length of space tours  has been taken as new parameters, the pseudo-differential operator $\mathcal{W}(\omega t)$ changes with the parameters.  We  should establish an equivalence relationship between pseudo-differential operators on different irrational torus, and prove that this relationship does not change with algebraic operation.
\begin{defn}\label{defn:same symbol}
Given two symbol $a  \in C^{\infty}(\T_{\beta_1}\times \mathbb{Z}), b\in C^{\infty}(\T_{\beta_2}\times \mathbb{Z})$, we say that the associated pseudo-differential operators $Op(a)$ and $Op(b)$  are in the same class, if
\begin{equation}
a(x,\xi)=b(\frac{\beta_2}{\beta_1}x,\xi).
\end{equation}
Namely, $a\approx b$.
\end{defn}

\begin{lem}\label{7.5}
Given the following four symbols, $a,b \in  C^{\infty}(\T_{\beta_1}\times \mathbb{Z})$ and $ c,d \in  C^{\infty}(\T_{\beta_2}\times \mathbb{Z})$. If $a \approx c$ and $b\approx d$, the composition of pseudo-differential operators $Op(a)\circ Op(b)$ and $Op(c) \circ Op(d)$ are  in the same class.
\end{lem}
\begin{proof}
Notice that $Op(a)\circ Op(b)=Op(a \sharp b)$, one gets
\begin{equation}
a \sharp b(x,\xi)=\sum_{j \in \Z} a(x, \xi+j)\hat{b}(j)e^{\mathrm{i} \frac{x}{\beta_1}\cdot j}.
\end{equation}
Let $Op(c) \circ Op(d)=Op(c \sharp d)$, we have
\begin{equation}
c \sharp d(x,\xi)=\sum_{j \in \Z} c(x, \xi+j)\hat{d}(j)e^{\mathrm{i} \frac{x}{\beta_2}\cdot j}.
\end{equation}
  From Definition \ref{defn:same symbol}, $b \approx d$ implies that $\hat{b}_j(\xi)= \hat{d}_j(\xi)$. Finally, we can get
   \begin{equation}
   a \sharp b(x,\xi) \approx c \sharp d(x,\xi).
   \end{equation}
\end{proof}
\textbf{The difference in reducing the order of perturbation:}\

The equation \eqref{eq2} can be rewritten as
\begin{equation}\label{eq5}
\mathrm{i}\partial_{t}u=  v \cdot \mathcal{K} u+ \mathcal{Q}(\omega t)u+\varepsilon \mathcal{W}(\omega t)[u],
\end{equation}
where $\mathcal{K} e^{\mathrm{i}j \cdot \frac{x}{\beta}}=|j| e^{\mathrm{i}j \cdot \frac{x}{\beta}}$, $v=\frac{1}{\beta} \in [1,2]$. $\mathcal{Q}$ is a  pseudo-differential operator of  order $-1$, and
$$\mathcal{Q}e^{\mathrm{i}j\frac{x}{\beta}}=\frac{c(\mathfrak{m},v,|j|)}{\langle j \rangle }e^{\mathrm{i} j \cdot \frac{x}{\beta}},$$
where
$$|c(\mathfrak{m},v,|j|)|^{\mathcal{L}ip} \leq \frac{1}{4}, \quad \forall j\in \mathbb{N} , \ 0<\mathfrak{m}<\frac{1}{4},\ v \in [1,2].$$

Moreover, we define  a  new parameter set $\tilde{\Omega}_{0,\alpha} \subseteq [1,2]^{d+1}$, where
\begin{equation}\label{nonres}
\tilde{\Omega}_{0,\alpha}=\Big\{\tilde{\omega}:=(\omega,v) \in \tilde{\Omega}:=[1,2]^{d+1}:|\omega \cdot\ell+ v \cdot  k| \geq \frac{\alpha}{(|\ell|+|k|)^{d+1}}, \ \forall (\ell,k) \in \mathbb{Z}^{d+1}\setminus \{0\}\Big\}.
\end{equation}
Hence, the Lemma \ref{4.3} can be replaced by the following Lemma.
 \begin{lem}\label{reduce2}Let $\mathcal{W}$ be an Hermitian operator and belongs to $ \mathcal{L}ip (\tilde{\Omega}_{0,\alpha},C^{\infty}(\mathbb{T}^d, OPS^{\eta})), \eta \leq 1$.
Then, the homological equation
\begin{equation}\label{h1}
\omega \cdot \partial_{\theta}  \mathcal{B}+\mathrm{i}[v\cdot\mathcal{K},\mathcal{B}]=\mathcal{W}-\langle \mathcal{W} \rangle
\end{equation}
with
\begin{equation}
\langle \mathcal{W} \rangle:= \frac{1}{(2\pi)^{d+1}}\int_{\mathbb{T}^d}\int_{\mathbb{T}}e^{\mathrm{i}\kappa\cdot \mathcal{K}} \mathcal{W} e^{-\mathrm{i}\kappa\cdot \mathcal{K}}d\kappa d\theta
\end{equation}
has a solution $\mathcal{B} \in \mathcal{L}ip (\tilde{\Omega}_{0,\alpha},C^{\infty}(\mathbb{T}^d, OPS^{\eta}))$. Moreover, the operator $\mathcal{B}$ is an Hermitian operator.
 \end{lem}
\begin{proof}
The proof is almost the same as Lemma \ref{4.3}.  The only difference is the homological equation \eqref{h1} is transformed to
\begin{equation}
\mathrm{i}(\omega\cdot \ell+v\cdot k)\hat{\mathcal{B}}_{\ell,k}=\hat{\mathcal{W}}_{\ell,k}, \quad  (\ell,k) \neq (0, 0).
\end{equation}
Using the non-resonance conditions \eqref{nonres}, we can obtain the conclusion by the same  way with Lemma \ref{4.3}.
\end{proof}
From Lemma \ref{7.5}  and Lemma \ref{reduce2}, we can repeat the process of Theorem \ref{3.0} without significant changes. We fix the number of regularization step as $$M:=1+4m.$$

After $M$ steps of regularization, the original equation \eqref{eq5} is transformed to
\begin{equation}\label{R2}
\mathrm{i}\omega \cdot \partial_{\theta}u=\mathbf{\Lambda}^0u+ \mathbf{P}^0u,
\end{equation}
where $\mathbf{\Lambda}^0=\mathcal{K}+\mathcal{Q}+\varepsilon \mathcal{Z}^M$ and  $\mathbf{P}^0 =\varepsilon\mathcal{W}^M$.
Denoting $(\mu_{j,n})_{n=1,2}$ as the eigenvalues of the  block $\Lambda_{j}$,
it has the asymptotic expression
\begin{equation}
\mu_{j,n}=v\cdot j+z(j,\mathfrak{m},v)+p_{j,n}(\tilde{\omega}), \quad j \in \mathbb{N}, \ n=1,2,
\end{equation}
where $|z(j,\mathfrak{m},v)|^{\mathcal{L}ip} \leq \frac{1}{4}$ and $|p_{j,n}|^{lip} \leq C \varepsilon$.\\

\textbf{The difference in KAM reducibility:}\

After finite times KAM iteration, the equation \eqref{R2} is converted to
\begin{equation}\label{R3}
\mathrm{i}\omega \cdot \partial_{\theta}u=\mathbf{\Lambda}^ku+ \mathbf{P}^{k}u.
\end{equation}
Furthermore, denoting $(\mu^k_{j,n})_{n=1,2}$ as the eigenvalues of the block $\Lambda^k_{j}$, it has the asymptotic expression
\begin{equation}
\mu^k_{j,n}=v j+z(j,\mathfrak{m},v)+p^k_{j,n}(\tilde{\omega}),\quad j \in \mathbb{N}, \ n=1,2,
\end{equation}
where $|p^k_{j,n}|^{lip} \leq C (\epsilon+\varepsilon)$.

Hence, we can define the non-resonance  set $\tilde{\Omega}_{k+1,\alpha}(\tilde{\omega})$ at the $k+1^{th}$ step reducibility as
\begin{equation}\label{nonres2}
\begin{split}
\widetilde{\Omega}_{k+1,\alpha}:=\Big\{\tilde{\omega} \in \tilde{ \Omega}_{k,\alpha}:&|\omega \cdot \ell+\mu^{k}_{i,n}-\mu^{k}_{j,n'}|\geq  \frac{\alpha}{N_k^{\tau}\langle i \rangle^{\sigma} \langle j \rangle^{\sigma}}, \\
&\forall i,j\in\mathbb{ N}, \quad |\ell| \leq N_k, \quad n, n'=1,2, \quad (\ell,i,j) \neq (0,i,i)
\Big\}.
\end{split}
\end{equation}

\begin{rem}In order to ensure that the gap of eigenvalues greater than some constant in Lemma \ref{5.7},
the condition \textbf{(C2)} is indispensable. However, by taking the length of space torus  as new parameters, there are some new phenomenons in considering the non-resonance conditions \eqref{nonres2}.
\end{rem}

\begin{lem}
$$meas(\tilde{\Omega}_{k,\alpha}\backslash \tilde{\Omega}_{k+1,\alpha}) \leq C \alpha N^{-1}_{k}.$$
\end{lem}
\begin{proof}
Considering the Lipschitz function $g(\tilde{\omega})$,
\begin{equation}
g(\tilde{\omega})=\omega \cdot \ell +\mu_{i,n}^k-\mu_{j,n'}^k=\tilde{\omega} \cdot(\ell, i-j)+z(i,\mathfrak{m},v)-z(j,\mathfrak{m},v)+p^k_{i,n}(\tilde{\omega})-p^k_{j,n'}(\tilde{\omega}).
\end{equation}
For any $(\ell,i-j) \neq 0$, we can write
\begin{equation}
\widetilde{\omega}=\frac{\ell,i-j}{|(\ell,i-j)|}s+\omega_1, \quad \omega_1 \in \mathbb{R}^{d+1}, \quad \omega_1 \cdot (\ell,i-j)=0,
\end{equation}
and
$$g(s)=|(\ell,i-j)| s + z(i,v(s))-z(j,v(s))+p^k_{i,n}(\tilde{\omega}(s))-p^k_{j,n'}(\tilde{\omega}(s)).$$

Subsequently, we have
\begin{equation}
\begin{split}
|g(s_1)-g(s_2)|\geq& |(\ell,i-j)| |s_1-s_2|-\Big(|z(i,v(s_1))-z(i,v(s_2))|+|z(j,v(s_1))-z(j,v(s_2))|\Big)\\
&-\Big(|p^k_{i,n}(\tilde{\omega}(s_1))-p^k_{i,n}(\tilde{\omega}(s_2))|+|p^k_{j,n'}(\tilde{\omega}(s_1))-p^k_{j,n'}(\tilde{\omega}(s_2))|\Big)\\
\geq& (1-\frac{1}{2}-C(\varepsilon+\epsilon))|s_1-s_2| \geq \frac{1}{4}|s_1-s_2|.
\end{split}
\end{equation}
which implies
\begin{equation}
meas \Big\{s\in \mathbb{R}: g_{\ell ijvv'}(s) <\frac{\alpha}{N^{\tau}_k \langle i\rangle^{\sigma} \langle j\rangle^{\sigma}}   \Big\} <\frac{4\alpha}{N^{\tau}_k \langle i\rangle^{\sigma} \langle j\rangle^{\sigma} }.
\end{equation}
The rest of proof is the same as Lemma \ref{5.7}.
\end{proof}

\section{Appendix B}
\subsection{Properties of pseudo-differential operators}
\begin{lem}(\cite{Bam019},Lemma A.1)\label{8.1}Let $\eta<1$ and $G \in C^{\infty}(\mathbb{T}^d,OPS^{\eta})$ be such that $G(\theta)+G^*(\theta)=0$ and let $e^{tG}$ be the flow of the autonomous PDE
$$\partial_{t}u=G(t)u, \quad t\in [-1,1] $$
$\mathbf{1}$: $\forall \sigma>0,$ $e^{tG} \in \mathcal{L}(\mathrm{H}^{\sigma})$.\\
$\mathbf{2}$: $\forall \sigma>0,$ $\forall \alpha \in \mathbb{N}^n$, $\partial^{\alpha}_{\theta}e^{tG}(\theta) \in  \mathcal{L}(\mathrm{H}^{\sigma},\mathrm{H}^{\sigma-\eta|\alpha|}).$\\
$\mathbf{3}$: If $G \in \mathcal{L}ip(\Omega, C^{\infty}(\mathbb{T}^d,OPS^{\eta}))$; $ \partial_{\theta}^{\alpha}e^{tG}(\theta,\omega) \in \mathcal{L}ip(\Omega, \mathcal{L}(\mathrm{H}^{\sigma},\mathrm{H}^{\sigma-\eta(|\alpha|+1)})) $, $\forall \sigma>0.  \alpha\in \mathbb{N}^d$.
\end{lem}

\begin{rem}\label{8.2}
Let $A(\theta) \in \mathcal{L}ip(\Omega, C^{\infty}(\mathbb{T}^d,OPS^m))$ and $G\in \mathcal{L}ip(\Omega, C^{\infty}(\mathbb{T}^d,OPS^\eta))$ with $\eta<1$. If $\forall j \in \mathbb{N}$, we define
$$Ad^0_{G}A=A, \quad Ad^{j+1}_{G}A=[G,Ad^{j}_{G}],$$
then $Ad^{j}_{G}A \in \mathcal{L}ip(\Omega, C^{\infty}(\mathbb{T}^d,OPS^{m-j(1-\eta)}))$.
\end{rem}

\begin{lem}(\cite{Bam019},Lemma A.2)\label{8.3}
Let $A(\theta) \in \mathcal{L}ip(\Omega, C^{\infty}(\mathbb{T}^d,OPS^m))$ and $G\in \mathcal{L}ip(\Omega, C^{\infty}(\mathbb{T}^d,OPS^\eta))$ with $\eta<1$  such that $G(\theta)+G^*(\theta)=0$. Then
\begin{equation}
e^{tG}Ae^{-tG} \in \mathcal{L}ip(\Omega, C^{\infty}(\mathbb{T}^d,OPS^m)).
\end{equation}
\end{lem}

\begin{rem}\label{8.4}From Theorem A.0.9 in \cite{T91}, one has that if $A(\theta) \in \mathcal{L}ip(\Omega, C^{\infty}(\mathbb{T}^d,OPS^m))$, then $\forall \alpha \in \mathbb{N}$,
$$e^{\mathrm{i}\kappa\cdot \mathcal{K}} A e^{-\mathrm{i}\kappa\cdot \mathcal{K}}, \quad \partial^{\alpha}_{\kappa}(e^{\mathrm{i}\kappa\cdot \mathcal{K}} A e^{-\mathrm{i}\kappa\cdot \mathcal{K}}) \in \mathcal{L}ip(\Omega, C^{\infty}(\mathbb{T}^d,OPS^m)).$$
\end{rem}
In the next Proposition, we essentially prove that pseudo-differential operators as in Definition \ref{defn:symbols20}  have  matrix presentation,  which belong to the classes $\mathcal{M}^{s,\mathcal{L}ip}_{s,s} $  extended from Definition \ref{3.40}.
\begin{prop}
\label{8.5}
Let $\mathrm{F} \in \mathcal{L}ip(\Omega,C^{\infty}(\mathbb{T}^d,OPS^{\mu}))$. For any $s>\frac{d+1}{2}$, the matrix of the operator $ \langle D\rangle^\gamma \mathrm{F}\langle D\rangle^\zeta$, $\gamma+\zeta+\mu\geq0$  belongs to $\mathcal{M}^{s,\mathcal{L}ip}_{s,s}$. Moreover, there exists $\sigma>0$, such that
\begin{equation}
\label{est:optoma}
\|\langle D\rangle^\gamma \mathrm{F}\langle D\rangle^\zeta\|_{s,s}^{s,\mathcal{L}ip} \leq C \ \chi^{0, \mathcal{L}ip}_{s+\sigma,s+\sigma}(F).
\end{equation}

\end{prop}
\begin{proof}
We start by proving the case $\gamma=\zeta=0$. Fix $s > \frac{d+1}{2}$, for any $m,n \in \mathbb{Z}$, we have
\begin{align}
\hat{\mathrm{F}}^{n}_{m}(\ell) &=\frac{1}{(2\pi)^{d}} \int_{\mathbb{T}^{d}}\hat{\mathrm{F}}(\ell)[e^{\mathrm{i}mx}]e^{-\mathrm{i}nx} dx\\
&=\frac{1}{(2\pi)^{d}} \int_{\mathbb{T}^{d}}f(x,m)(\ell)e^{\mathrm{i}(m-n)x}dx.
\end{align}
For the case $m \neq n$. Integrating by parts $\tilde{s}$ times in $x$, with $\tilde{s}=\lfloor s\rfloor+2$,
For any $n,m \in \mathbb{N}$, $n \neq m$, $\ell \in \mathbb{Z}^d$, one gets
\begin{equation}
\|\hat{\mathrm{F}}^{[n]}_{[m]}(\ell)\| \leq \sup_{|k|=m,|k^{'}|=n} |\hat{\mathrm{F}}^{k^{'}}_{k}(\ell)|\leq \frac{1}{|m-n|^{\tilde{s}}} \chi^{0}_{\tilde{s}}(\hat{\mathrm{F}}(\ell))
\end{equation}

For the case $m = n$, we can prove $\|\hat{\mathrm{F}}^{[n]}_{[n]}(\ell)\| \leq \chi^{0}_{\tilde{s}}(\hat{\mathrm{F}}(\ell))$ in a similar way.
Thus, we can get
\begin{equation}
\|\mathrm{F}\|^{s}_{s,s} \leq C \chi^{0}_{\tilde{s},s}(\mathrm{F}) \leq C^* \chi^{0}_{\tilde{s},\tilde{s}}(\mathrm{F}).
\end{equation}
For the other cases, the operator $\langle D \rangle^{\gamma} \mathrm{F } \langle D \rangle^{\zeta}$ belongs to $\mathcal{L}ip(\Omega,C^{\infty}(\mathbb{T}^d,OPS^{0}))$, so we have
\begin{equation}
\|\langle D \rangle^{\gamma} \mathrm{F } \langle D \rangle^{\zeta}\|^{s}_{s,s} \leq C^* \chi^{0}_{\tilde{s},\tilde{s}}(\langle D \rangle^{\gamma} \mathrm{F } \langle D\rangle^{\zeta} ) \leq C\chi^{\mu}_{\tilde{s},\tilde{s}}(\mathrm{F } ).
\end{equation}
\end{proof}

\begin{lem}\label{OPS}The operator $(-\partial_{xx}+\mathfrak{m}^2)^{\frac{1}{2}}-(-\partial_{xx})^{\frac{1}{2}}$ is a  pseudo-differential
operator of order $-1$.
\end{lem}
\begin{proof}
From Theorem 1 in \cite{CdV79}, there exists a pseudo-differential operator  $\mathrm{K}$ of order $-1$, commuting with $-\partial_{xx}$ , such that
\begin{equation}\label{c11}
Spec[(-\partial_{xx}+\mathfrak{m}^2)^{\frac{1}{2}}+\mathrm{K}]\subseteq \mathbb{N}+c,\quad c\in \mathbb{R}.
\end{equation}
Since $-\partial_{xx}$ and  $\mathrm{K}$ can be diagonalized simultaneously, one can obtain that there exists an orthonormal basis $\Phi_{j},\Phi_{-j}$ of space $E_{j}:=span \{e^{\mathrm{i}jx}, e^{-\mathrm{i}jx}\}$, such that
$$\mathrm{K}\Phi_{j}=\eta_{j}\Phi_{j},\quad \mathrm{K}\Phi_{-j}=\eta_{-j}\Phi_{-j}. $$
Also, there exists an absolute constant $C$ such that
\begin{equation}\label{c22}
\eta_{j} \leq \frac{C}{\langle j\rangle}.
\end{equation}
 Hence, we see that
 \begin{equation}\label{c33}
 [(-\partial_{xx}+\mathfrak{m}^2)^{\frac{1}{2}}+\mathrm{K}]\Phi_{j}=\lambda_j\Phi_{j}=(|j|+\frac{c(\mathfrak{m},|j|)}{\langle j \rangle}+\eta_{j})\Phi_{j}.
 \end{equation}

From \eqref{c11}, \eqref{c22} and \eqref{c33}, we get that $c=0$ and $\lambda_j=|j|$, if $|j|$ is large enough. Furthermore, there exists a $N \in \mathbb{N}$, if $|j| \geq N$,  $\Phi_{j}$ and $\Phi_{-j}$ can be determined as $e^{\mathrm{i}jx}$ and  $e^{-\mathrm{i}jx}$. Finally, we can construct two symbol $k^1,k^2$ as
$$ k^1(x,j)=\left\{
\begin{aligned}
&0, \quad \quad \quad  \quad \quad  \quad \quad \quad  \quad  \quad |j|\leq N,\\
&\eta_j=(j^2+\mathfrak{m}^2)^{\frac{1}{2}}-|j|,\quad \ |j| \geq N,
\end{aligned}
\right.$$
$$
k^2(x,j)=\left\{
\begin{aligned}
&(j^2+\mathfrak{m}^2)^{\frac{1}{2}}-|j|, \quad  \quad |j|\leq N,\\
&0,  \quad \quad \quad  \quad \quad  \quad \quad  \quad  \  \ |j| \geq N.
\end{aligned}
\right.
$$

From the above argument, we know that $Op(k^1) \in OPS^{-1}$, $Op(k^2)$ is a finite rank operator and belongs to $OPS^{-\infty}$. We see that $(-\partial_{xx}+\mathfrak{m}^2)^{\frac{1}{2}}-(-\partial_{xx})^{\frac{1}{2}} = Op(k^1)+Op(k^2)$, which belongs to $ OPS^{-1}$.
\end{proof}

\subsection{Properties of Hermitian matrix}\

In this section, we recall some well known facts about Hermitian  operator in the finite dimension Hilbert space $\mathcal{H}$. Let $\mathcal{H}$ be a
finite dimensional Hilbert space of dimension $n$ equipped by the inner product
$(, )_\mathcal{H}$. For any Hermitian  operator $A$, we order its eigenvalues as
$spec(A) := {\lambda_1(A) \leq \lambda_2(A) \leq \cdots \leq \lambda_n(A)}.$

\begin{prop}\label{8.6}(Weyl's Perturbation Theorem)(\cite{R97}, Theorem III.2.1) Let $A$ and $B$ be Hermitian matrices. Then
\begin{equation}
|\lambda_k(A)-\lambda_k(B)|\leq \|A-B\|_{\mathcal{L}^2(\mathcal{H})},\  \forall k \in 1,\cdots,n.
\end{equation}
\end{prop}
\begin{prop}\label{8.7}(\cite{R97}, Theorem VII.2.8) Let $A$ and $B$ be Hermitian matrices, and let $\delta=dist(\sigma(A),\sigma(B))$. Then the solution $X$ of the equation $AX-XB=Y$ satisfies the inequality
\begin{equation}
\|X\|_{\mathcal{L}^2(\mathcal{H})}  \leq \frac{C}{\delta} \|Y\|_{\mathcal{L}^2(\mathcal{H})}.
\end{equation}
\end{prop}

\section*{Acknowledgements}
The authors are grateful to the referee for his/her valuable comments, which greatly improve
the original manuscript of this paper. The work is supported by the Natural Science Foundation of the Jiangsu Higher Education Institutions of China (Grant No.19KJB110025), School Foundation of Yangzhou University(Grant No.2019CXJ009) and by the National Natural Science Foundation of China (Grant No.12071254).

\bibliographystyle{abbrv} 

\begin{thebibliography}{10}
\bibitem{bal19}
P.~Baldi, M.~Berti, E.~Haus, R.~Montalto.
\newblock Time quasi-periodic gravity water waves in finite depth.
\newblock {\em Invent. Math.}, 214, 739-911, 2018.

\bibitem{Baldi2}
P.~Baldi, M.~Berti, E.~Haus, R.~Montalto.
\newblock  KAM for quasi-linear and fully nonlinear forced perturbations of Airy equation.
\newblock  {\em Mathematische. Annalen.}, 359, 471-536, 2014.


\bibitem{Bam01}
D.~Bambusi and S.~Graffi.
\newblock Time Quasi-periodic unbounded perturbations of Schr\"odinger operators and KAM methods.
\newblock {\em Comm. Math. Phys.}, 219: 465-480, 2001.
\bibitem{Bam18}
D.~Bambusi.
\newblock Reducibility of 1-d Schr\"odinger equation with time quasiperiodic unbounded perturbation,I.
\newblock{\em Trans. Amer. Math. Soc.}, 370: 1823-1865, 2018.

\bibitem{Bam171}
D.~Bambusi.
\newblock Reducibility of 1-d Schr\"odinger equation with time quasiperiodic unbounded perturbation, II.
\newblock{\em Comm. Math. Phys.}, 353: 353-378, 2017.

\bibitem{Bam19}
D.~Bambusi and R.~Montalto.
\newblock Reducibility of 1-d Schr\"odinger equation with time quasiperiodic unbounded perturbation, III.
\newblock{\em  J. Math. Phys.}, 59, 2018.

\bibitem{Bam018}
D. Bambusi, B. Gr\'ebert, A. Maspero, and D. Robert.
\newblock Reducibility of the quantum harmonic oscillator in d-dimensions with polynomial time-dependent perturbation.
\newblock {\em Anal. PDE.,} 11: 775-799, 2018.


\bibitem{Bam019}
D.~Bambusi, B.~Langella  and  R.~Montalto.
\newblock Reducibility of non-resonant transport equation on $\T^d$ with unbound perturbation.
\newblock { \em Ann.Hernri. poincar\'e.}, 20: 1893-1929, 2019.


\bibitem{Bam17}
D.~Bambusi, B.~Gr\'ebert, A.~Maspero and D.~Robert.
\newblock Growth of Sobolev norms for abstract linear Schr\"odinger Equations.
\newblock {\em to appear in  J. Eur. Math. Soc.}, arXiv:1706.09708v2.


\bibitem{berti}
M.~Berti, L.~Corsi, M.~Procesi.
\newblock  An abstract Nash Moser theorem and quasi-periodic solutions for NLW and NLS on compact Lie groups and Homogeneous manifolds.
\newblock {\em Comm. Math. Phys.}, 334: 1413-1454, 2015.

\bibitem{berti2}
M.~Berti, R.~Montalto.
\newblock Quasi-Periodic Standing Wave Solutions of Gravity-Capillary Water Waves.
\newblock {\em Mem. Amer. Math. Soc.}, 263 : 2020.

\bibitem{R97}
R.~Bhatia.
\newblock Matrix Analysis.
\newblock{\em Springer-Verlag New York}, 1997.

\bibitem{CdV79}
Y.~Colin de Verdi\'ere.
\newblock Sur le spectre des op\'erateurs elliptiques \`a bicaract\'eristiques toutes p\'eriodiques.
\newblock{\em Comment. Math. Helv.}, 54: 508-522, 1979.

\bibitem{Car1990}
R.~Carmona and W.~C.~Masters.
\newblock Relativistic Schr\"odinger operator: Asymptotic Behavior of Eigenfunction.
\newblock {\em J. Funct. Anal.}, 91: 117-142, 1990.



\bibitem{L.H 09}
L.~Eliasson and S.~Kuksin.
\newblock  On reducibility of Schr\"odinger equation with quasi periodic in time potential.
\newblock {\em Comm. Math. Phys.}, 286: 125-135, 2009.




\bibitem{L.M2019}
L.~Franzoi and A.~Maspero.
\newblock Reducibility for a fast-driven linear Klein-Gordon equation.

\newblock{\em Annali. di. Matematica. Pura. ed. Applicata.}, 198: 1407-1439, 2019.



\bibitem{R.F2018}
R.~Feola, F.~Giuliani, R.~Montalto and M.~Procesi.
\newblock Reducibility of first order linear operators on tori via Moser's theorem.
\newblock  {\em J. Funct. Anal.}, 276: 932-970, 2019.

\bibitem{F.G2019}
R.~Feola, B.~Gr\'ebert.
\newblock Reducibility of Schr\"odinger equation on sphere.
\newblock {\em to appear in Int. Math. Res. Notices.}, arXiv:1905.11964.
\bibitem{F.G20192}
R.~Feola, F.~Giuliani, M.~Procesi.
\newblock Reducibility for a class of weakly dispersive linear operators arising from the Degasperis Procesi equation.
\newblock {\em Dyn. Partial Differ. Equ.}, 16: 25-94, 2019.

\bibitem{F.G2020}
R.~Feola, B.~Gr\'ebert, T.~Nguyen.
\newblock Reducibility of Schr\"odinger equation on a Zoll manifold with unbounded potential.
\newblock Preprint, arXiv:1910.10657.
\bibitem{F.G20201}
R. Feola, F. Giuliani, and M. Procesi.
\newblock Reducible KAM tori for the Degasperis-Procesi equation.
\newblock{\em  Commun. Math. Phys.}, 3: 1681-1759, 2020.

\bibitem{B12}
B.~Gr\'ebert and  L.~Thomann.
\newblock KAM for the Quantum Harmonic Oscillator.
\newblock {\em Commun. Math. Phys.}, 307: 383-427, 2011.

\bibitem{B24}
B.~Gr\'ebert and E.~Paturel.
\newblock On reducibility of Quantum Harmonic oscillator on $\mathbb{R}^d$ with quasiperiodic in time potential.
\newblock {\em  Ann. Fac. Sci. Toulouse.}, 6:977-1014, 2019.

\bibitem{Ku93}
S.~Kuksin.
\newblock Nearly integrable infinite-dimensional Hamiltonian systems,volume 1556 of Lecture Notes in Mathematics.
\newblock {\em Springer-Verlag. Berlin.}, 1993.

\bibitem{Ku97}
S.~Kuksin.
\newblock  On small-denominators equations with large variable coefficients.
\newblock{\em Z. Angew. Math. Phys.}, 48: 262-271, 1997.

\bibitem{Liang19}
Z.~Wang and Z.~Liang.
\newblock Reducibility of quantum harmonic oscillator on Rd with differential and quasi-periodic in time potential.
\newblock{\em J. Differential. Equations.}, 267: 3355-3395, 2019.

\bibitem{Liu09}
J.~Liu and X.~Yuan.
\newblock Spectrum for quantum Duffing oscillator and small-divisor equation with large-variable coefficient.
\newblock {\em Comm. Pure. Appl. Math.}, 63: 1145-1172, 2010.

\bibitem{Mas18}
A.~Maspero.
\newblock Lower bounds on the growth of Sobolev norms in some linear time dependent Schr\"odinger equations.
\newblock {\em Math.Res.Lett.}, 26: 1197-1215, 2019.

\bibitem{Mon2018}
R.~Montalto.
\newblock Quasi-periodic solutions of forced Kirchhoff equation.
\newblock{\em NoDEA Nonlinear. Differential. Equations. Appl.}, 24: 2017.


\bibitem{Mo2019}
R.~Montalto.
\newblock A reducibility result for a class of linear  wave equation  on $\T^d$.
\newblock {\em Int. Math. Res. Notices.}, 6: 1788-1862, 2019.




\bibitem{Mo20191}
R.~Montalto.
\newblock Growth of Sobolev norms for time dependent periodic Schr\"odinger equations with sublinear dispersion.
\newblock{\em J. Differential. Equations.}, 266: 4953-4996, 2019.

\bibitem{s19}
Y.~Shi.
\newblock Analytic solutions of nonlinear elliptic equations on rectangular tori.
\newblock{\em J. Differential. Equations.}, 267: 5576-5600, 2019.

\bibitem{S19}
Y.~Sun, J.~Li and  B.~Xie.
\newblock Reducibility for wave equations of finitely smooth potential with periodic boundary conditions.
\newblock{\em J. Differential. Equations.}, 266: 2762-2804, 2019.

\bibitem{T91}
M.~Taylor.
\newblock Pseudodifferential operators and nonlinear PDE, volume 100 of Progress in Mathematics.
\newblock Birkh\"auser Boston, Inc.,Boston, MA, 1991.
\bibitem{Wan08}
W.~Wang.
\newblock Pure point spectrum of the Floquet Hamiltonian for the quantum harmonic oscillator under time quasi-periodic perturbations.
\newblock {\em Commun. Math. Phys.}, 277: 459-496,2008.

\bibitem{W16}
Z.~Wang and Z.~Liang.
\newblock Reducibility of 1D quantum harmonic oscillator perturbed by a quasiperiodic potential with logarithmic decay.
 \newblock {\em Nonlinearity}, 30: 1405-1448, 2017.

\bibitem{Y2006}
X.~Yuan.
\newblock A varint of KAM theorem with applications to nonlinear wave equations of higher dimension.
\newblock Preprint, $https://web.ma.utexas.edu/mp_arc/c/06/06-44.pdf$, 2006.
\end{thebibliography}

\end{document}